\documentclass[12pt,a4paper]{amsart}
\usepackage{geometry}
\usepackage{amsmath,amssymb,amsfonts,amscd,amsthm,wasysym,color,enumitem,indentfirst,graphicx,booktabs,titletoc}
\usepackage[bookmarksnumbered,colorlinks,linktocpage,plainpages]{hyperref}
\hypersetup{pdfstartview={FitH}}

\newtheorem{thmx}{Theorem}

\numberwithin{equation}{section}
\newtheorem{theorem}{Theorem}[section]

\theoremstyle{definition}
\newtheorem{definition}[theorem]{Definition}

\newtheorem{remark}[theorem]{Remark}

\allowdisplaybreaks[4]

\makeatletter
\@namedef{subjclassname@2020}{\textup{2020} Mathematics Subject Classification}
\makeatother

\begin{document}

\title[Complex Geodesics and Complex Monge--Amp\`{e}re Equations]
{Complex geodesics and complex Monge--Amp\`{e}re equations with boundary singularity II}

\author{Xieping Wang}
\address{CAS Wu Wen-Tsun Key Laboratory of Mathematics and School of Mathematical Sciences, University of Science and Technology of China, Hefei 230026, Anhui, People's Republic of China}
\email{xpwang008@ustc.edu.cn}

\thanks{The author was partially supported by NSFC grants 12001513 and 12371083, NSF of Anhui Province grant 2008085QA18, and the Fundamental Research Funds for the Central Universities.}

\subjclass[2020]{Primary 32F45, 32W20, 32U35; Secondary 32F17, 35J96, 35B30}
\keywords{Strongly linearly convex domains, complex geodesics, complex Monge--Amp\`{e}re equations, Monge--Amp\`{e}re foliations}

\dedicatory{To Li and Rui'an}

\begin{abstract}
We study the parameter dependence of complex geodesics with prescribed boundary value and direction on bounded strongly linearly convex domains. As an important application we establish a quantitative relationship between the regularity of the pluricomplex Poisson kernel of such a domain, which is a solution to a homogeneous complex Monge--Amp\`{e}re equation with boundary singularity, and the regularity of the boundary of the domain. Our results greatly improve the previous results of Chang--Hu--Lee and Bracci--Patrizio in this direction.
\end{abstract}
\maketitle

\tableofcontents
\dottedcontents{subsection}[0pt]{\filright}{-33pt}{0pt}

\section{Introduction}
This paper is a continuation of our previous work with Huang \cite{Huang-Wang}, in which we proved the uniqueness of complex geodesics with prescribed boundary data on bounded strongly linearly convex domains with $C^3$-smooth boundary and solved a homogeneous complex Monge--Amp\`{e}re (HCMA for short) equation with prescribed {\it boundary singularity} on such domains. The purpose of the present paper is to investigate the parameter dependence of such complex geodesics and the higher regularity of solutions to the aforementioned complex Monge--Amp\`{e}re equation.

Our motivation for this work stems mainly from the following considerations. On the one hand, complex geodesics have proved to be very useful in several complex variables and CR geometry (see \cite{Lempert81, Lempert86, Abatebook, Burns-Krantz, Huang-Pisa94, Bracci-JEMS}, to name only a few), and further exploration may lead to new applications. On the other hand, although it has long been known that the regularity of solutions to HCMA equations is generally at most $C^{1,\,1}$  (see \cite{Gam-Sibony} for a well-known example, and also \cite{BF-counter}), higher regularity than $C^{1,\,1}$---which is quite important in applications---is still not well-understood, even in certain special contexts with rich geometric structure. Encouragingly, much progress has been made in this regard over the last few decades, among which we mention particularly \cite{Lempert81, Lempert83, Lempert86, Donaldson, Bracci-MathAnn, LV13}. The higher regularity of solutions to the HCMA equation we consider in this paper will be established through a detailed study of the associated complex foliation by complex geodesic discs initiated from a fixed boundary point, where the equation has a simple singularity. Our approach is largely inspired by the work of Lempert \cite{Lempert81, Lempert86}, Chang--Hu--Lee \cite{Chang--Hu--Lee88} and Huang \cite{Huang-Pisa94}. In connection with the subject of the present paper and the approach we adopt here, we also mention that complex (or rather Monge--Amp\`{e}re) foliations and HCMA equations are closely related and have played a significant role in proving many fascinating results that illustrate the interplay between pluripotential theory, several complex variables and complex geometry. For a historical overview and references in this regard, we refer the reader to \cite{BK-CPAM, Patriziobook, RWN_Survey}.

We now elaborate on our main results. To begin with, we fix some necessary notation and terminology. Let $\Delta$ denote the open unit disc in $\mathbb C$, and let $\Omega\subset\mathbb C^n,\,n>1$, be a bounded strongly linearly convex domain with $C^3$-smooth boundary. A complex geodesic of $\Omega$ is by definition a holomorphic mapping from $\Delta$ to $\Omega$ that realizes the Kobayashi distance between any two distinct points in its image. Significantly different from real geodesics in Riemannian geometry, complex geodesics are global holomorphic objects whose existence is a very subtle problem. In his celebrated paper \cite{Lempert81} and later work \cite{Lempert84}, Lempert showed that complex geodesics exist in great abundance on $\Omega$ and enjoy sufficiently high boundary regularity and a number of other good properties.

Let us now further assume that $\partial \Omega$ is $C^{m,\,\alpha}$-smooth, where $m\geq 4$ is an integer and $0<\alpha<1$. Throughout the paper, we always make this assumption unless otherwise specified. Set
\begin{equation}\label{defn:fiberbundle}
   S_{\partial\Omega}:=\big\{(p, v)\in \partial\Omega \times \mathbb C^n\!: v\in L_p\big\},
\end{equation}
where
   $$
   L_p:=\big\{v\in\mathbb C^n\!: |v|=1,\, \langle v, \nu_p\rangle>0\big\},
   $$
$\nu_p$ denotes the unit outward normal to $\partial\Omega$ at $p$ and $\langle\;,\,\rangle$ denotes the standard Hermitian inner product on $\mathbb C^n$. Then $S_{\partial\Omega}$ is a $C^{m-1,\,\alpha}$-smooth real submanifold of $\mathbb C^{2n}$ with dimension $4n-3$.
According to \cite{Huang-Pisa94, Huang-Wang} (see also \cite{Chang--Hu--Lee88} for prior partial results), for every $(p, v)\in S_{\partial\Omega}$ there exists a unique\footnote{In fact, as shown in \cite{Huang-Wang}, the $C^3$-regularity of $\partial\Omega$ is sufficient for the purpose here. We also emphasize that the proof of the uniqueness part requires a careful analysis with new ideas (although the strongly convex case is much easier).} complex geodesic $\varphi_{p,\, v}$ of $\Omega $ such that $\varphi_{p,\, v}(1)=p$, $\varphi_{p,\, v}'(1)=\langle v, \nu_p\rangle v$ and
\begin{equation}\label{dual-nor-condition}
   \left.\frac{d}{d\theta}\right|_{\theta=0}|\varphi^{\ast}_{p,\, v}(e^{i\theta})|=0,
\end{equation}
where  $\varphi^{\ast}_{p,\, v}$ is the dual mapping of $\varphi_{p,\, v}$ (see Subsection \ref{Kobayashi-geodesic} below for a precise definition). Following \cite{Huang-Wang}, we refer to such a unique $\varphi_{p,\, v}$ as the {\it preferred complex geodesic of $\Omega$ associated to  $(p, v)\in S_{\partial\Omega}$}.

We can now state our first main result.

\begin{thmx}\label{thm:parameter dependence}
Let $\Omega$ and $S_{\partial\Omega}$ be as above. Then for every $k\in\{2,\ldots,m-2\}$ and $\varepsilon\in (0,\alpha)$, the mapping
   $$
   \Phi\!: S_{\partial\Omega}\to\mathcal{O}(\Delta,\,\mathbb C^n)
   \cap C^{k,\,\varepsilon}(\overline{\Delta},\,\mathbb C^n), \quad (p, v)\mapsto \varphi_{p,\, v}
   $$
is locally $C^{m-k-1,\,\alpha-\varepsilon}$-smooth.
\end{thmx}

Here $\mathcal{O}(\Delta,\,\mathbb C^n)$ denotes the set of all holomorphic mappings from $\Delta$ to $\mathbb C^n$, and $C^{k,\,\varepsilon}(\overline{\Delta},\,\mathbb C^n)$ is the usual H\"{o}lder space on $\overline{\Delta}$ with norm defined in a standard way.

Combining Theorem \ref{thm:parameter dependence} with some fundamentally critical results for extremal mappings (with respect to the Kobayashi--Royden metric) proved by Huang \cite{Huang-Illinois94, Huang-Pisa94}, we then obtain the following result concerning the differentiable dependence of preferred complex geodesics on their parameters.

\begin{thmx}\label{thm:joint-regularity}
The mapping
   $$
   \widetilde{\Phi}\!: S_{\partial\Omega}\times\overline{\Delta}\to \mathbb C^n,
   \quad (p, v, \zeta)\mapsto \varphi_{p,\,v}(\zeta)
   $$
is locally $C^{m-3,\,\alpha-\varepsilon}$-smooth for all $\varepsilon\in (0,\alpha)$.
\end{thmx}

Next, we turn to some consequences of Theorem \ref{thm:joint-regularity}, particularly an application to the higher regularity of solutions to the HCMA equation considered in \cite{Bracci-MathAnn, Bracci-Trans, Bracci-Adv, Huang-Wang}. For this we need to recall the so-called boundary spherical representation for $\Omega$,  constructed by Huang and the author in \cite{Huang-Wang}. Denote by $\mathbb B^n$ the open unit ball in $\mathbb C^n$. It is easy to check that for every $(q, v)\in S_{\partial \mathbb B^n}$,  the mapping
\begin{equation*}\label{ball-CG}
  \eta_{q,\, v}\!:\Delta \to\mathbb B^n, \quad \zeta \mapsto q+(\zeta-1)\langle v, q\rangle v
\end{equation*}
is the associated preferred complex geodesic of $\mathbb B^n$. The boundary spherical representation $\Psi_p\!:\overline{\Omega}\to\overline{\mathbb B}^n$ is then defined as $\Psi_p(p):=\nu_p$ and
\begin{equation}\label{defn:BSR}
   \Psi_p(z):=\nu_p+(\zeta_{z,\,p}-1)\langle v_{z,\,p},\, \nu_p\rangle v_{z,\,p},\quad
   z\in\overline{\Omega}\!\setminus\!\{p\},
\end{equation}
where $\zeta_{z,\,p}:=\varphi_{p,\, v_{z,\,p}}^{-1}(z)$ and $v_{z,\,p}\in L_p$ is the only data such that the associated preferred complex geodesic $\varphi_{p,\, v_{z,\,p}}$ of $\Omega$ passes through $z$, i.e.,  $z\in \varphi_{p,\, v_{z,\,p}}(\overline{\Delta})$. In this way $\Psi_p$ is a bijection with inverse $\Psi_p^{-1}$ given by $\Psi^{-1}_p(\nu_p)=p$ and
\begin{equation}\label{defn:inv-BSR}
   \Psi^{-1}_p(z)=\varphi_{p,\,v_{w,\,\nu_p}}(\zeta_{w,\,\nu_p}),\quad w\in\overline{\mathbb B}^n\!\setminus\!\{\nu_p\},
\end{equation}
where $(v_{w,\,\nu_p},\,\zeta_{w,\,\nu_p})\in L_p\times\overline{\Delta}$ is the only data such that $\eta_{\nu_p,\,v_{w,\,\nu_p}}(\zeta_{w,\,\nu_p})=w$; more explicitly,
\begin{equation}\label{ball-splitting}
   v_{w,\,\nu_p}=-\frac{1-\langle \nu_p, w\rangle}{|1-\langle \nu_p, w\rangle|}\frac{w-\nu_p}{|w-\nu_p|},
   \quad \zeta_{w,\,\nu_p}=1-\frac{|w-\nu_p|^2}{|1-\langle \nu_p, w\rangle|^2}(1-\langle w, \nu_p\rangle).
\end{equation}
It was also shown in \cite{Huang-Wang} that $\Psi_p\!:\overline{\Omega}\to\overline{\mathbb B}^n$ is a homeomorphism and, among other things,
   $$
   \Psi(z,\,p):=\Psi_p(z),\quad (z,\,p)\in\overline{\Omega}\times\partial\Omega,
   $$
is continuous (provided $\partial\Omega$ is only $C^3$-smooth; see \cite[Section 3]{Huang-Wang} for details).
To gain a better understanding of the relationship between the regularity of $\Psi$ and that of the boundary $\partial\Omega$, we prove in this paper the following

\begin{thmx}\label{thm:regularity-BSR}
For every $p\in \partial\Omega$ and sufficiently small $\delta>0$, the mapping
   $$
   \Psi_p\!:\overline{\Omega}\!\setminus\!B(p, \delta)\to
            \overline{\mathbb B}^n\!\setminus\!\Psi_p\big(B(p, \delta)\big)
   $$
is a $C^{m-3,\,\alpha-\varepsilon}$-smooth diffeomorphism for all $\varepsilon\in (0,\alpha)$, where $B(p, \delta)$  denotes the open ball of radius $\delta$ centered at $p$ in $\mathbb C^n$. Moreover,
   $$
   \Psi\in C^{m-3,\,\alpha-\varepsilon}\big((\overline{\Omega}\times\partial\Omega)
   \!\setminus\!{\rm dist}_{\delta}\, (\overline{\Omega}\times\partial\Omega)\big)
   $$
for all $\varepsilon\in (0,\alpha)$, where
   $$
   {\rm dist}_{\delta}\, (\overline{\Omega}\times\partial\Omega):=\big\{(z,\, p)\in \overline{\Omega}\times\partial\Omega\!: |z-p|<\delta\big\}.
   $$
\end{thmx}

This result strengthens a similar (and deep) result of Chang--Hu--Lee \cite{Chang--Hu--Lee88}, drawing a stronger conclusion under a much weaker condition, and thus may lead to more applications. Relying heavily on Lempert's deformation theory (see \cite{Lempert81}) and the Chern--Moser--Vitushkin normal form theory, Chang--Hu--Lee's  original argument requires $\partial\Omega$ to be sufficiently smooth and is also technically involved (and quite sketchy at some points). In this paper, with the great help of Theorem \ref{thm:joint-regularity} and some new observations (see Theorems \ref{thm:RH-prob2} and \ref{thmA:normalization} below), we can adapt the strategy of Lempert \cite{Lempert86} to the present context. However, we have to overcome some difficulties caused by the bad behavior of the normalization condition \eqref{dual-nor-condition} under generic holomorphic coordinate transformations of $\Omega$ (see Section \ref{sect:proof of thm C} for details, and also compare with the proof of
\cite[Theorem 5.1]{Lempert86}).  It should also be emphasized that the approach in this paper only requires that $\partial\Omega$ be $C^{4,\,\alpha}$-smooth, and the treatment here seems much simpler than that in \cite{Chang--Hu--Lee88}.

We are now in a position to consider the following HCMA equation on $\Omega$ with $p\in\partial\Omega$:
\begin{equation}\label{eq:Bdry-MA}
\begin{cases}
u\in {\rm Psh}(\Omega)\cap
L_{{\rm loc}}^{\infty}(\Omega), \\
(dd^c u)^n=0 \quad \quad  \;\;\; \ \ \,  {\rm on}\,\ \Omega, \\
\lim_{z\to x}u(z)=0 \;\;\; \!\: \quad {\rm for}\,\ x\in\partial\Omega\!\setminus\!\{p\}, \\
u(z)\approx -|z-p|^{-1} \:\:\,\,\,\!\: {\rm as}
\ \ z\to p\, \ {\rm nontangentially},
\end{cases}
\end{equation}
which was first studied by Bracci-Patrizio in \cite{Bracci-MathAnn} where $\Omega$ is further assumed to be strongly convex with $C^m$-smooth boundary for $m\ge 14$ (here ${\rm Psh}(\Omega)$  denotes the set of all plurisubharmonic functions on $\Omega$).
Put
  $$
  P_{\Omega}(z, p)=-\frac{1-|\Psi_p(z)|^2}{|1-\langle \Psi_p(z), \nu_p\rangle|^2},
  $$
which is continuous on $(\overline{\Omega}\times\partial\Omega)\!\setminus\!{\rm diag }\, \partial\Omega$
where
  $$
  {\rm diag }\, \partial\Omega:=\big\{(z, z)\in \mathbb C^{2n}\!: z\in\partial\Omega\big\}.
  $$
By virtue of \cite[Theorem 1.3]{Huang-Wang} (and its proof), we know that for every $p\in\partial\Omega$ the function $P_{\Omega}(\,\cdot\,,\, p)$ solves the above equation. Now as an immediate consequence of Theorem \ref{thm:regularity-BSR} we obtain the following

\begin{thmx}\label{thm:regularity-HCMA}
For every sufficiently small $\delta>0$,
   $$
   P_{\Omega}\in C^{m-3,\,\alpha-\varepsilon}\big((\overline{\Omega}\times\partial\Omega)
   \!\setminus\!{\rm dist}_{\delta}\, (\overline{\Omega}\times\partial\Omega)\big)
   $$
for all $0<\varepsilon<\alpha$.
\end{thmx}

The particular interest in studying equation \eqref{eq:Bdry-MA} lies in the fact that it is a boundary analogue of the following interior one on $\Omega$ with $w\in\Omega$:
\begin{equation}\label{eq:Inter-MA}
  \begin{cases}
   u\in {\rm Psh}(\Omega)\cap
   L_{{\rm loc}}^{\infty}(\Omega\!\setminus\!\{w\}), \\
   (dd^c u)^n=0 \qquad \qquad  \quad \!\quad\ \ \,\, {\rm on}\,\ \Omega \!\setminus\! \{w\}, \\
   \lim_{z\to x}u(z)=0  \qquad \qquad \quad\!  \: {\rm for}\,\ x\in\partial\Omega, \\
   u(z)-\log|z-w|=O(1)\ \ \;  {\rm as} \ \ z\to w,\\
  \end{cases}
\end{equation}
which has been extensively studied in the literature; see \cite{Lempert81, Klimek-Green, Demailly87, BD-counter, GuanMAeqn, Blocki00, Blocki05} among others. The solution to equation \eqref{eq:Inter-MA} can easily be shown to be unique and is now known as the pluricomplex Green function. It shares many properties with the classical Green function in one complex variable and has profound applications, especially when it appears as a weight function in H\"{o}rmander's $L^2$ theory of the $\overline{\partial}$-operator. We note that there has been a great deal of work along this line since the work of Diederich--Ohsawa \cite{Die-Ohsawa95}, and here we refer the reader only to \cite{Blocki-survey, Chen-Survey} for historical background and to \cite{Chen-APDE} for recent developments.

The study of equation \eqref{eq:Bdry-MA} was initiated in \cite{Bracci-MathAnn, Bracci-Trans} and continued very recently in \cite{Huang-Wang, Bracci-Adv}. It turns out that the function $P_{\Omega}(\,\cdot\,,\, p)$ constructed as above bears a strong resemblance to the classical Poisson kernel in real potential theory (see \cite{Bracci-MathAnn, Bracci-Trans}), and has already found important applications in the theory of holomorphic semigroups (see \cite{Bracci-JEMS}) and others. Nevertheless, it is still an open question, very interesting at least from the PDE point of view, whether positive multiples of $P_{\Omega}(\,\cdot\,,\, p)$ are the only solutions to equation \eqref{eq:Bdry-MA}. Under certain analytic and/or geometric conditions, a partial answer to this question can be found in \cite{Bracci-Trans}.

At present, it seems difficult to answer the above question in full generality. To the best of our knowledge, the well-known comparison principle for the complex Monge--Amp\`{e}re operator due to Bedford--Taylor \cite{Bedford82}, as well as its variants in different contexts, remains so far the only available tool for proving the uniqueness of solutions to complex Monge--Amp\`{e}re equations of various types (for more information on this see, for example, the recent monograph \cite{GZ-MAeqnbook} and the references therein). Unfortunately, such a crucial tool fails to work in the present context because of the presence of {\it boundary singularity} in equation \eqref{eq:Bdry-MA}. It seems likely that Theorem \ref{thm:regularity-HCMA} could be useful in seeking an affirmative answer to the above question. We hope to focus on this in future work and expect to find further applications of the results obtained in this paper and its predecessor \cite{Huang-Wang}.

The paper is organized as follows. In Section \ref{sect:preliminaries} we first collect some  necessary background material on complex geodesics of bounded strongly linearly convex domains and then establish some preliminary results that we need in the subsequent sections of the paper. Theorems \ref{thm:parameter dependence} and \ref{thm:joint-regularity} are proved in Sections \ref{sect:proof of thm A} and \ref{sect:proof of thm B}, respectively. Section \ref{sect:proof of thm C} is devoted to the proof of Theorem \ref{thm:regularity-BSR}. In the course of the proof we need a result of Lempert on performing a canonical change of coordinates for strongly linearly convex domains along a given complex geodesic disc. We include a detailed proof of this result in Appendix \ref{supplement} for the reader's convenience.

\smallskip
\noindent {\bf Acknowledgements.} The author is greatly indebted to Xiaojun Huang for many enlightening discussions related to this work and for his constant support and encouragement over the years. The author also wishes to express his gratitude to L\'{a}szl\'{o} Lempert, whose work \cite{Lempert81} on complex geodesics has always been an important source of ideas and motivation for this work. Thanks also go to W{\l}odzimierz Zwonek for his comments on the proof of Theorem \ref{thmA:normalization} and for bringing \cite{KZ16} to the author's attention. Last but not least, the author thanks the anonymous referees for their suggestions and comments that helped improve the exposition of the paper.


\section{Preliminaries}\label{sect:preliminaries}
\subsection{The Kobayashi distance and complex geodesics}\label{Kobayashi-geodesic}
We briefly recall the definitions of the Kobayashi--Royden metric and the Kobayashi distance, as well as some well-known results concerning complex geodesics on bounded strongly linearly convex domains in $\mathbb C^n$. For a complete picture of the subject,  the interested reader may consult \cite{Abatebook, Lempert81, Lempert84, KW13}.

Let $\Delta$ denote the open unit disc in $\mathbb C$. The {\it Kobayashi--Royden metric} $\kappa_{\Omega}$ on a domain $\Omega\subset\mathbb C^n$ is  the pseudo-Finsler metric defined by
$$
\kappa_{\Omega}(z, v):=\inf\big\{\lambda>0\,|\;\exists \;\varphi\in \mathcal{O}(\Delta,\, \Omega)\!: \varphi(0)=z,\, \varphi'(0)=\lambda^{-1}v \big\}, \quad (z, v)\in \Omega\times\mathbb C^n,
$$
where $\mathcal{O}(\Delta,\, \Omega)$ denotes the set of all holomorphic mappings from $\Delta$ to $\Omega$.
The {\it Kobayashi distance} $k_{\Omega}$ on $\Omega$ is then defined to be the integrated form of $\kappa_{\Omega}$, that is
$$
k_{\Omega}(z, w)=\inf_{\gamma\in \Gamma}\int_{0}^{1}\kappa_{\Omega}(\gamma(t), \,\gamma'(t))dt,
\quad (z, w)\in \Omega\times \Omega,
$$
where $\Gamma$ is the set of all $C^1$-smooth curves $\gamma\!:[0, 1] \rightarrow \Omega$ such that $\gamma(0)=z$ and $\gamma(1)=w$. For the open unit disc  $\Delta\subset\mathbb C$, $k_{\Delta}$ coincides with the classical \textit{Poincar\'{e} distance}, i.e.,
$$
k_{\Delta} (\zeta_1, \zeta_2)=\tanh^{-1}\bigg|\frac{\zeta_1-\zeta_2}{1-\zeta_1\overline{\zeta}_2}\bigg|,
\quad (\zeta_1, \zeta_2)\in \Delta\times\Delta.
$$
A holomorphic mapping $\varphi\!:\Delta\to\Omega$ is called a {\it complex geodesic} of $\Omega$ if it is an isometry between $k_{\Delta}$ and $k_{\Omega}$, i.e.,
\begin{equation*}\label{isometry}
k_{\Omega}(\varphi(\zeta_1),\,\varphi(\zeta_2))=k_{\Delta}(\zeta_1, \zeta_2)
\end{equation*}
for all $\zeta_1$, $\zeta_2\in\Delta$.

We proceed with one more definition.

\begin{definition}
A domain $\Omega$ in $\mathbb C^n,\, n>1$, is called  {\it strongly linearly convex} if it has $C^2$-smooth boundary and admits a $C^2$-defining function $\rho\!:\mathbb C^n\to\mathbb R$ whose real Hessian is positive definite on the complex tangent space of $\partial\Omega$, i.e.,
  $$
  \sum_{j,\, k=1}^n\frac{\partial^2 \rho}{\partial z_j\partial\overline
  z_k}(p)v_j\overline{v}_k
  >\bigg|\sum_{j,\, k=1}^n\frac{\partial^2 \rho}{\partial z_j\partial z_k}(p)v_jv_k\bigg|
  $$
for all $p\in\partial\Omega$ and nonzero $v=(v_1,\ldots,v_n)\in T_p^{1,\, 0}\partial\Omega$.
\end{definition}

At first sight, the class of strongly linearly convex domains appears to be rather restrictive. However, as recently discovered, this is not the case. There are bounded strongly linearly convex domains with real-analytic boundary, which are not biholomorphic to convex domains and cannot even be exhausted by biholomorphic images of such domains; see \cite{Pflug-Zwonek}.

We also record the following intrinsic characterization of complex geodesics on strongly linearly convex domains, which is due to Lempert \cite{Lempert84} (see also \cite{KW13} for a very clear exposition) and will be used in this paper in a crucial way.

\begin{theorem}\label{thm:Lempert1}
Let $\Omega\subset\mathbb C^n, \, n>1$, be a bounded strongly linearly convex domain with $C^2$-smooth boundary,  and let $\nu$ be the unit  outward normal vector field of $\partial\Omega$. Then a holomorphic mapping $\varphi\!:\Delta\to\Omega$ is a complex geodesic of $\Omega$ if and only if it satisfies the following conditions:
\begin{enumerate}[leftmargin=2.0pc, parsep=4pt]
\item [{\rm(i)}]
$\varphi\in C^{1/2}(\overline{\Delta},\,\mathbb C^n)$;
\item [{\rm(ii)}]
$\varphi$ is proper, i.e., $\varphi(\partial\Delta)\subset\partial\Omega$;
\item [{\rm(iii)}]
there exists a positive function $\mu\in C^{1/2}(\partial\Delta,\,\mathbb R)$ such that the mapping
$$
\partial\Delta\ni \zeta\mapsto \zeta\mu(\zeta)\overline{\nu\circ\varphi(\zeta)}\in\mathbb C^n
$$
extends holomorphically to $\Delta$; and
\item [{\rm(iv)}]
the winding number of the function
$$
\partial\Delta\ni \zeta\mapsto \langle z-\varphi(\zeta),\, \nu\circ\varphi(\zeta)\rangle
$$
is zero for some {\rm(}and hence all\,{\rm)} $z\in\Omega$.
\end{enumerate}
\end{theorem}

Regarding the last condition, observe that  the (strong) linear convexity of $\Omega$ implies that the function under consideration does not vanish  anywhere, so its winding number is well-defined.

For a complex geodesic $\varphi$ of $\Omega$, the resulting holomorphic mapping as in {\rm(iii)}, which we denote by $\varphi^{\ast}$, can be normalized so that $\langle\varphi', \, \overline{\varphi^{\ast}}\rangle=1$ on $\overline{\Delta}$. The mapping $\varphi^{\ast}$ satisfying such a normalization condition is unique and usually called the {\it dual mapping} of $\varphi$. Lempert also proved that if $\partial\Omega$ is further $C^{m,\,\alpha}$-smooth for some $m\ge 2$ and $\alpha\in(0, 1)$, then both $\varphi$ and $\varphi^{\ast}$ are  $C^{m-1,\,\alpha}$-smooth up to the boundary.

\subsection{Riemann--Hilbert problems}\label{Lem:RH-problem}
Given $m\in\mathbb N$ and $\alpha\in(0, 1)$, we let $\mathcal{O}^{m,\,\alpha}(\partial\Delta,\,\mathbb C^n)$ denote the subspace of $C^{m,\,\alpha}(\partial\Delta,\,\mathbb C^n)$ consisting of mappings that admit holomorphic extension to $\Delta$. Throughout the paper we shall always not distinguish between a mapping $\varphi\in \mathcal{O}^{m,\,\alpha}(\partial\Delta,\,\mathbb C^n)$ and its holomorphic extension to $\Delta$, which allows us to speak of its interior values.

Of crucial importance in the proofs of Theorems \ref{thm:parameter dependence} and \ref{thm:regularity-BSR} are the following two results, which can be thought of as the generalizations of a result of Lempert  \cite{Lempert86} from the $0$-jet (i.e., one-point) case to the $1$-jet and two-point cases, respectively.

\begin{theorem}\label{thm:RH-prob1}
Suppose $m\in\mathbb N^{\ast},\, \alpha\in(0, 1)$, and let $H,\, S\in C^{m,\,\alpha}(\partial\Delta,\,\mathbb C^{n\times n})$ be such that $H(\zeta)$ is Hermitian,  $S(\zeta)$ is symmetric and
\begin{equation}\label{eq:cond-RH}
   v^tH(\zeta)\bar{v}>|v^tS(\zeta)v|
\end{equation}
for all $\zeta\in\partial\Delta$ and $v\in\mathbb C^n\!\setminus\!\{0\}$. Then for every $f\in C^{m,\,\alpha}(\partial\Delta,\,\mathbb C^n)$ and $(\zeta_0,\, z_0,\,v_0)\in \overline{\Delta}\times\mathbb C^n\times\mathbb C^n$, there exists a unique $g\in \mathcal{O}^{m,\,\alpha}(\partial\Delta,\,\mathbb C^n)$ such that $g(\zeta_0)=z_0$, $g'(\zeta_0)=v_0$ and
   $$
   H\overline{(g/{\rm Id_{\partial\Delta}})}+Sg/{\rm Id_{\partial\Delta}}+f\in
   \mathcal{O}^{m,\,\alpha}(\partial\Delta,\,\mathbb C^n).
   $$
\end{theorem}

Here and henceforward we denote by ${\rm Id_{\partial\Delta}}$ the identity map of $\partial\Delta$.

\begin{proof}
We divide the argument into two cases depending on whether $\zeta_0$ lies within  $\Delta$ or not.

\medskip
\noindent  {\bf Case 1:} $\zeta_0\in \Delta$.
\smallskip

If $\zeta_0=0$, the result is an immediate consequence of \cite[Lemma 4.2]{Lempert86}. Otherwise, we take a $\sigma\in {\rm Aut}(\Delta)$ such that $\sigma(0)=\zeta_0$. Then there exists a positive smooth function $\mu\!:\partial\Delta\to \mathbb R$ such that
$\partial\Delta\ni\zeta\mapsto\zeta\mu(\zeta)/\sigma(\zeta)$ extends to a nowhere-vanishing holomorphic function $h$ on $\overline{\Delta}$. In fact, it suffices to take
   $$
   \mu(\zeta):=\frac{\sigma(\zeta)}{\zeta\sigma'(\zeta)},\quad \zeta\in \partial\Delta.
   $$
Now the desired result follows by applying the one in the case of $\zeta_0=0$ to
the pair $(\mu^{-1}H\circ\sigma,\, \mu^{-1}S\circ\sigma)$ instead of $(H,\,S)$.

\medskip
\noindent  {\bf Case 2:} $\zeta_0\in \partial\Delta$.
\smallskip

Let $f$ be as described. According to \cite[Lemma 4.2]{Lempert86}, there exists a unique $g^{\ast}\in \mathcal{O}^{m,\,\alpha}(\partial\Delta,\,\mathbb C^n)$ such that $g^{\ast}(0)=0$ and
   $$
   H\overline{g^{\ast}}+Sg^{\ast}+f\in\mathcal{O}^{m,\,\alpha}(\partial\Delta,\,\mathbb C^n).
   $$
Let $\{e_1,\ldots, e_n\}$ denote the standard basis of $\mathbb C^n$. In the same way, for every  $k\in\{1,\ldots,n\}$  one can also find a unique pair $g^r_k,\,g^i_k\in \mathcal{O}^{m,\,\alpha}(\partial\Delta,\,\mathbb C^n)$ such that
   $$
   g^r_k(0)=e_k,\quad  H\overline{g^r_k}+Sg^r_k\in\mathcal{O}^{m,\,\alpha}(\partial\Delta,\,\mathbb C^n)
   $$
and
   $$
   g^i_k(0)=ie_k,\quad H\overline{g^i_k}+Sg^i_k\in\mathcal{O}^{m,\,\alpha}(\partial\Delta,\,\mathbb C^n).
   $$
Similarly, there exists a unique pair $h^r_k,\,h^i_k\in \mathcal{O}^{m,\,\alpha}(\partial\Delta,\,\mathbb C^n)$ with $h^r_k(0)=h^i_k(0)=0$ such that
\begin{equation*}
\begin{split}
   H\overline{(h^r_k+e_k/{\rm Id_{\partial\Delta}})}+S(h^r_k+e_k/{\rm Id_{\partial\Delta}})
   &=H\overline{h^r_k}+Sh^r_k+{\rm Id_{\partial\Delta}}He_k+Se_k/{\rm Id_{\partial\Delta}}\\
   &\in\mathcal{O}^{m,\,\alpha}(\partial\Delta,\,\mathbb C^n)
\end{split}
\end{equation*}
and
\begin{equation*}
\begin{split}
   H\overline{(h^i_k+ie_k/{\rm Id_{\partial\Delta}})}+S(h^i_k+ie_k/{\rm Id_{\partial\Delta}})
   &=H\overline{h^i_k}+Sh^i_k-i{\rm Id_{\partial\Delta}}He_k+iSe_k/{\rm Id_{\partial\Delta}}\\
   &\in\mathcal{O}^{m,\,\alpha}(\partial\Delta,\,\mathbb C^n).
\end{split}
\end{equation*}
Now a mapping $g:=(g_1, \ldots,g_n)^t
\in\mathcal{O}^{m,\,\alpha}(\partial\Delta,\,\mathbb C^n)$ satisfies
   $$
   H\overline{(g/{\rm Id_{\partial\Delta}})}+Sg/{\rm Id_{\partial\Delta}}+f\in
   \mathcal{O}^{m,\,\alpha}(\partial\Delta,\,\mathbb C^n)
   $$
if and only if the mapping
   $$
   h:=\frac{g-g(0)}{{\rm Id_{\partial\Delta}}}-g^{\ast}-
      \sum_{k=1}^{n}\Big(\big({\rm Re}\,g_k(0)\big)h^r_k+\big({\rm Im}\,g_k(0)\big)h^i_k\Big)
     \in\mathcal{O}^{m,\,\alpha}(\partial\Delta,\,\mathbb C^n)
   $$
satisfies
   $$
   H\overline{h}+Sh\in\mathcal{O}^{m,\,\alpha}(\partial\Delta,\,\mathbb C^n).
   $$
As $h(0)=g'(0)$, it follows from the uniqueness part of \cite[Lemma 4.2]{Lempert86} that
   $$
   h=\sum_{k=1}^{n}\Big(\big({\rm Re}\,g'_k(0)\big)g^r_k+\big({\rm Im}\,g'_k(0)\big)g^i_k\Big).
   $$
Consequently, $g$ takes the form
\begin{equation}\label{eq:solutionformula}
   g=g(0)+{\rm Id_{\partial\Delta}}\bigg\{g^{\ast}+\sum_{k=1}^{n}\Big(\big({\rm Re}\,g'_k(0)\big)g^r_k
     +\big({\rm Im}\,g'_k(0)\big)g^i_k+\big({\rm Re}\,g_k(0)\big)h^r_k+
      \big({\rm Im}\,g_k(0)\big)h^i_k\Big)\bigg\}.
\end{equation}
This implies that the existence and uniqueness of a mapping $g$ with the properties stated in the theorem is equivalent to saying that $g(0)$ and $g'(0)$ are uniquely determined by the prescribed boundary data.

To complete the proof, we now only need to show that the latter is indeed true. In view of  equality \eqref{eq:solutionformula}, it suffices to show that the vectors
$$
\begin{pmatrix}
    g^r_k(\zeta_0)\\
    (g^r_k)'(\zeta_0)\\
\end{pmatrix},
\quad
\begin{pmatrix}
    g^i_k(\zeta_0)\\
    (g^i_k)'(\zeta_0)\\
\end{pmatrix},
\quad
\begin{pmatrix}
    h^r_k(\zeta_0)+\zeta^{-1}_0e_k\\
    (h^r_k)'(\zeta_0)-\zeta^{-2}_0e_k\\
\end{pmatrix},
\quad
\begin{pmatrix}
   h^i_k(\zeta_0)+i\zeta^{-1}_0e_k\\
   (h^i_k)'(\zeta_0)-i\zeta^{-2}_0e_k\\
\end{pmatrix},
$$
$k=1,\ldots, n$, are linearly independent over $\mathbb R$. This amounts to proving that $\varphi=0$ is the only element of $\mathcal{O}^{m,\,\alpha}(\partial\Delta,\,\mathbb C^n)$ that satisfies
\begin{equation*}\label{eq:solution-unique1}
   H\overline{(\varphi/{\rm Id_{\partial\Delta}})}+S\varphi/{\rm Id_{\partial\Delta}}\in
   \mathcal{O}^{m,\,\alpha}(\partial\Delta,\,\mathbb C^n),\quad \varphi(\zeta_0)=\varphi'(\zeta_0)=0.
\end{equation*}
Clearly, it suffices to consider the case when $m=1$. Taking into account the invariance of the last condition under rotations we may assume that $\zeta_0=1$. Also by virtue of \cite[Th\'{e}or\`{e}me B]{Lempert81}, we may further assume that $H=I_n$. Then condition \eqref{eq:cond-RH} means precisely that for every $\zeta\in\partial\Delta$ the operator norm $\|S(\zeta)\|$ of the matrix $S(\zeta)$ is less than one.

Now suppose that $\varphi\in \mathcal{O}^{1,\,\alpha}(\partial\Delta,\,\mathbb C^n)$ satisfies $\varphi(1)=\varphi'(1)=0$ and
\begin{equation}\label{eq:solution-varphi}
   \widetilde{\varphi}:=\overline{\varphi/{\rm Id_{\partial\Delta}}}+S\varphi/{\rm Id_{\partial\Delta}}
   \in\mathcal{O}^{1,\,\alpha}(\partial\Delta,\,\mathbb C^n).
\end{equation}
Set
   $$
   q:=\sup\big\{p\geq 0\!: ({\rm Id_{\partial\Delta}}-1)^{-2}\varphi
      \in L^p(\partial\Delta,\, \mathbb C^n)\big\}.
   $$
Clearly $q\geq(1-\alpha)^{-1}>1$. We claim that $q=\infty$. If this were not the case, there would exist a $Q>q$ such that $(1-\alpha)Q<q$. Then
\begin{align}\label{ineq:L^Q-intergability1}
\begin{split}
 \int_{0}^{2\pi}&\bigg|\frac{(S(e^{i\theta})-S(1))\varphi(e^{i\theta})}{(e^{i\theta}-1)^2}\bigg|^Qd\theta
   \leq \int_{0}^{2\pi}\bigg(\frac{\|S(e^{i\theta})-S(1)\||\varphi(e^{i\theta})|}
        {|e^{i\theta}-1|^2}\bigg)^Qd\theta\\
     &=\int_{0}^{2\pi}\bigg(\frac{\|S(e^{i\theta})-S(1)\|}{|e^{i\theta}-1|^{\alpha}}\bigg)^Q
     \bigg|\frac{\varphi(e^{i\theta})}{e^{i\theta}-1}\bigg|^{\alpha Q}
     \bigg|\frac{\varphi(e^{i\theta})}{(e^{i\theta}-1)^2}\bigg|^{(1-\alpha)Q}d\theta<\infty,
\end{split}
\end{align}
since the first two factors of the integrand in the last integral are bounded and the last one is integrable.
Note that the function ${\rm Id_{\partial\Delta}}/({\rm Id_{\partial\Delta}}-1)^2$ takes values in $\mathbb R\cup \{\infty\}$. By \eqref{eq:solution-varphi} we then have the  decomposition
\begin{equation}\label{eq:L^Q-decomposition1}
   \frac{\,(S-S(1))\varphi\,}{({\rm Id_{\partial\Delta}}-1)^2}
   =\frac{\,{\rm Id_{\partial\Delta}}\widetilde{\varphi}-S(1)\varphi\,}{({\rm Id_{\partial\Delta}}-1)^2}
    -\overline{\, \frac{\varphi}{({\rm Id_{\partial\Delta}}-1)^2}\,}.
\end{equation}
In view of this,  \eqref{ineq:L^Q-intergability1} and a classical result of Riesz that the Hilbert transform maps $L^Q(\partial\Delta, \,\mathbb C)$ into itself, we conclude that both terms on the right-hand side of \eqref{eq:L^Q-decomposition1} belong to $L^Q(\partial\Delta, \,\mathbb C^n)$. This contradicts the fact that $Q>q$, and the claim follows.

Now we see in particular that both $({\rm Id_{\partial\Delta}}-1)^{-2}\varphi$ and
$({\rm Id_{\partial\Delta}}-1)^{-2}S\varphi$ belong to $L^2(\partial\Delta,\, \mathbb C^n)$. Let
  $$
  \sum_{k=0}^{\infty}a_k\zeta^k \quad {\rm and} \quad \sum_{k=-\infty}^{\infty}b_k\zeta^k
  $$
be their Fourier expansions, respectively. By \eqref{eq:solution-varphi} we have $b_{-k}=-\overline{a}_k$ for all $k\geq0$. If $\varphi\neq 0$, an application of the Parseval identity and the fact that $\sup_{\partial\Delta}\|S\|<1$ would give
\begin{align*}\label{ineq:RH-contradiction1}
\begin{split}
    \frac{1}{2\pi}&\int_{0}^{2\pi}\bigg|\frac{\varphi(e^{i\theta})}{(e^{i\theta}-1)^2}\bigg|^2d\theta
    =\sum_{k=0}^{\infty}|a_k|^2\leq\sum_{k=-\infty}^{\infty}|b_k|^2\\
    &=\frac{1}{2\pi}\int_{0}^{2\pi}\bigg|S(e^{i\theta})\frac{\varphi(e^{i\theta})}{(e^{i\theta}-1)^2}\bigg|^2
     d\theta<\frac{1}{2\pi}\int_{0}^{2\pi}\bigg|\frac{\varphi(e^{i\theta})}{(e^{i\theta}-1)^2}\bigg|^2d\theta,
\end{split}
\end{align*}
yielding a contradiction. It follows therefore that $\varphi=0$, as desired. This proves the theorem.
\end{proof}

Much as with the $1$-jet case, we also have the following 2-point case.

\begin{theorem}\label{thm:RH-prob2}
Suppose $m\in\mathbb N,\, \alpha\in(0, 1)$, and let $H,\, S\in C^{m,\,\alpha}(\partial\Delta,\,\mathbb C^{n\times n})$ be such that $H(\zeta)$ is Hermitian, $S(\zeta)$ is symmetric and
  $$
   v^tH(\zeta)\bar{v}>|v^tS(\zeta)v|
  $$
for all $\zeta\in\partial\Delta$ and $v\in\mathbb C^n\!\setminus\!\{0\}$. Then for every $f\in C^{m,\,\alpha}(\partial\Delta,\,\mathbb C^n)$ and $(\zeta_0,\, z_0),\, (\xi_0,\, w_0)\in \overline{\Delta}\times\mathbb C^n$ with $\zeta_0\neq\xi_0$, there exists a unique $g\in \mathcal{O}^{m,\,\alpha}(\partial\Delta,\,\mathbb C^n)$ such that $g(\zeta_0)=z_0$, $g(\xi_0)=w_0$  and
   $$
   H\overline{(g/{\rm Id_{\partial\Delta}})}+Sg/{\rm Id_{\partial\Delta}}+f\in
   \mathcal{O}^{m,\,\alpha}(\partial\Delta,\,\mathbb C^n).
   $$
\end{theorem}

\begin{proof}
The argument is essentially the same as that in the proof of Theorem \ref{thm:RH-prob1}. Here, we  only consider the case when $\zeta_0,\,\xi_0\in \partial\Delta$ (the remaining is relatively easy and similar to Case 1 there, so the details are omitted). As before, the mapping $g$ we are looking for must be of the form \eqref{eq:solutionformula}. Thus the problem reduces to proving that $g(0)$ and $g'(0)$ are uniquely determined by the prescribed boundary data. In view of equality \eqref{eq:solutionformula} again, it suffices to show that the  vectors
$$
\begin{pmatrix}
    g^r_k(\zeta_0)\\
    g^r_k(\xi_0)\\
\end{pmatrix},
\quad
\begin{pmatrix}
    g^i_k(\zeta_0)\\
    g^i_k(\xi_0)\\
\end{pmatrix},
\quad
\begin{pmatrix}
    h^r_k(\zeta_0)+\zeta_0^{-1}e_k\\
    h^r_k(\xi_0)+\xi_0^{-1}e_k\\
\end{pmatrix},
\quad
\begin{pmatrix}
   h^i_k(\zeta_0)+i\zeta^{-1}_0e_k\\
   h^i_k(\xi_0)+i\xi^{-1}_0e_k\\
\end{pmatrix},
$$
$k=1,\ldots, n$, are linearly independent over $\mathbb R$. This amounts to proving that $\psi=0$ is the only element of $\mathcal{O}^{m,\,\alpha}(\partial\Delta,\,\mathbb C^n)$ that satisfies
\begin{equation*}\label{eq:solution-unique2}
   H\overline{(\psi/{\rm Id_{\partial\Delta}})}+S\psi/{\rm Id_{\partial\Delta}}\in
   \mathcal{O}^{m,\,\alpha}(\partial\Delta,\,\mathbb C^n),\quad \psi(\zeta_0)=\psi(\xi_0)=0.
\end{equation*}
It suffices to consider the case when $m=0$.

For simplicity, we may assume as before that $H=I_n$. Suppose
$\psi\in \mathcal{O}^{\alpha}(\partial\Delta,\,\mathbb C^n)$ satisfies $\psi(\zeta_0)=\psi(\xi_0)=0$ and
\begin{equation}\label{eq:solution-psi}
   \widetilde{\psi}:=\overline{\psi/{\rm Id_{\partial\Delta}}}+S\psi/{\rm Id_{\partial\Delta}}
   \in\mathcal{O}^{\alpha}(\partial\Delta,\,\mathbb C^n).
\end{equation}
Set
   $$
   q:=\sup\Big\{p\geq 0\!: ({\rm Id_{\partial\Delta}}-\zeta_0)^{-1}
       ({\rm Id_{\partial\Delta}}-\xi_0)^{-1}\psi\in L^p(\partial\Delta,\,\mathbb C^n)\Big\}.
   $$
Clearly $q\geq(1-\alpha)^{-1}>1$. We claim that $q=\infty$. Suppose on the contrary that $q<\infty$ and choose a $Q>q$ such that $(1-\alpha)Q<q$. Then
\begin{align}\label{ineq:L^Q-intergability2}
\begin{split}
     \int_{0}^{2\pi}&\bigg|\bigg(\frac{e^{i\theta}-\xi_0}{\zeta_0-\xi_0}\big(S(e^{i\theta})-S(\zeta_0)\big)
          +\frac{\zeta_0-e^{i\theta}}{\zeta_0-\xi_0}\big(S(e^{i\theta})-S(\xi_0)\big)\bigg)
          \frac{\psi(e^{i\theta})}{(e^{i\theta}-\zeta_0)(e^{i\theta}-\xi_0)}\bigg|^Qd\theta\\
     \leq\, & \frac{1}{|\zeta_0-\xi_0|}\int_{0}^{2\pi}
          \bigg(\frac{\|S(e^{i\theta})-S(\zeta_0)\|}{|e^{i\theta}-\zeta_0|}+
          \frac{\|S(e^{i\theta})-S(\xi_0)\|}{|e^{i\theta}-\xi_0|}
          \bigg)^Q|\psi(e^{i\theta})|^Qd\theta\\
     \leq\, &\frac{2^{Q-1}}{|\zeta_0-\xi_0|}\int_{0}^{2\pi}
          \bigg\{\bigg(\frac{\|S(e^{i\theta})-S(\zeta_0)\|}{|e^{i\theta}-\zeta_0|}\bigg)^Q+
          \bigg(\frac{\|S(e^{i\theta})-S(\xi_0)\|}{|e^{i\theta}-\xi_0|}\bigg)^Q
          \bigg\}|\psi(e^{i\theta})|^Qd\theta\\
     =\, &\frac{2^{Q-1}}{|\zeta_0-\xi_0|}\int_{0}^{2\pi}
          \bigg(\frac{\|S(e^{i\theta})-S(\zeta_0)\|}{|e^{i\theta}-\zeta_0|^{\alpha}}\bigg)^Q|
          \psi(e^{i\theta})|^{\alpha Q}
          \bigg|\frac{\psi(e^{i\theta})}{e^{i\theta}-\zeta_0}\bigg|^{(1-\alpha)Q}d\theta\\
         &+\frac{2^{Q-1}}{|\zeta_0-\xi_0|}\int_{0}^{2\pi}
          \bigg(\frac{\|S(e^{i\theta})-S(\xi_0)\|}{|e^{i\theta}-\xi_0|^{\alpha}}\bigg)^Q|
          \psi(e^{i\theta})|^{\alpha Q}
          \bigg|\frac{\psi(e^{i\theta})}{e^{i\theta}-\xi_0}\bigg|^{(1-\alpha)Q}d\theta\\
      <\, &\infty.
\end{split}
\end{align}
On the other hand, by \eqref{eq:solution-psi} we have the  decomposition
\begin{equation}\label{eq:L^Q-decomposition2}
\begin{split}
   \bigg(&\frac{{\rm Id_{\partial\Delta}}-\xi_0}{\zeta_0-\xi_0}\big(S-S(\zeta_0)\big)
        +\frac{\zeta_0-{\rm Id_{\partial\Delta}}}{\zeta_0-\xi_0}\big(S-S(\xi_0)\big)
   \bigg)\frac{\psi}{({\rm Id_{\partial\Delta}}-\zeta_0)({\rm Id_{\partial\Delta}}-\xi_0)}\\
   =\,&\frac1{({\rm Id_{\partial\Delta}}-\zeta_0)({\rm Id_{\partial\Delta}}-\xi_0)}
     \bigg\{{\rm Id_{\partial\Delta}}\widetilde{\psi}-
     \bigg(\frac{{\rm Id_{\partial\Delta}}-\xi_0}{\zeta_0-\xi_0}S(\zeta_0)
        +\frac{\zeta_0-{\rm Id_{\partial\Delta}}}{\zeta_0-\xi_0}S(\xi_0)\bigg)\psi\bigg\}\\
     &-\overline{\,\frac{\zeta_0\xi_0\psi}{({\rm Id_{\partial\Delta}}-\zeta_0)
      ({\rm Id_{\partial\Delta}}-\xi_0)}\,},
\end{split}
\end{equation}
from which together with \eqref{ineq:L^Q-intergability2} and the aforementioned Riesz theorem we conclude that both terms on the right-hand side of \eqref{eq:L^Q-decomposition2} belong to $L^Q(\partial\Delta, \,\mathbb C^n)$. This contradicts the fact that $Q>q$, and the claim follows.

We now obtain in particular
   $$
   \frac{\psi}{({\rm Id_{\partial\Delta}}-\zeta_0)
      ({\rm Id_{\partial\Delta}}-\xi_0)}\in L^2(\partial\Delta,\,\mathbb C^n).
   $$
Also by \eqref{eq:solution-psi} we have
   $$
   \overline{\,\frac{\zeta_0\xi_0\psi}{({\rm Id_{\partial\Delta}}-\zeta_0)
   ({\rm Id_{\partial\Delta}}-\xi_0)}\,}+S\frac{\psi}{({\rm Id_{\partial\Delta}}-\zeta_0)
   ({\rm Id_{\partial\Delta}}-\xi_0)}
   =\frac{{\rm Id_{\partial\Delta}}\widetilde{\psi}}{({\rm Id_{\partial\Delta}}-\zeta_0)
   ({\rm Id_{\partial\Delta}}-\xi_0)},
   $$
which combined with a similar argument at the end of the proof of Theorem~\ref{thm:RH-prob1} implies  $\psi=0$.  This completes the proof.
\end{proof}

\section{Proof of Theorem \ref{thm:parameter dependence}}\label{sect:proof of thm A}

Let us start with the following

\begin{remark}\label{rem:observations}
We make here a few observations to be used in the proof of the theorem, all of which are very simple but clarify matters greatly.

\begin{enumerate}[leftmargin=1.8pc, parsep=4pt]
\item  [{\rm (i)}]
Let $\Omega\subset\mathbb C^n, \, n>1$, be a domain with $C^{m,\,\alpha}$-smooth boundary, where $m\in\mathbb N^{\ast}$ and $0<\alpha<1$. Then the set $S_{\partial\Omega}$ given in \eqref{defn:fiberbundle} is a $C^{m-1,\,\alpha}$-smooth real submanifold of $\mathbb C^{2n}$ with dimension $4n-3$. In fact, $S_{\partial\Omega}$ is even a fiber bundle over $\partial\Omega$ with fiber
   $$
   \mathbb S^{2n-2}_+:=\big\{(\sqrt{1-|\hat{v}|^2},\, \hat{v})\in\mathbb R\times\mathbb C^{n-1}\!: |\hat{v}|<1\big\}.
   $$
To see this, we denote by $U(n)$ the unitary group of degree $n$ and by $^t$ the transpose of matrices. For every $p_0\in\partial\Omega$ we can find a small neighborhood $U_{\partial\mathbb B^n}\subset \partial\mathbb B^n$ of $\nu_{p_0}$  and a smooth  mapping
   $$
   U_{\partial\mathbb B^n}\to U(n),\quad v\mapsto\gamma_v
   $$
such that $\gamma_v(e_1)=v$ for all $v\in U_{\partial\mathbb B^n}$, where $e_1:=(1,0,\ldots, 0)^t$. Set
   $$
   U_{\partial\Omega}:=\nu^{-1}(U_{\partial\mathbb B^n})\ni p_0\quad \mbox{and}\quad
   U_{S_{\partial\Omega}}:=\big\{(p, v)\in S_{\partial\Omega}\!: \ p\in U_{\partial\Omega}\big\},
   $$
where $\nu\!:\partial\Omega\to\partial\mathbb B^n$ denotes the unit outward normal vector field of $\partial\Omega$. Then it is evident that the mapping
\begin{equation}\label{defn:local-trivialization}
  U_{S_{\partial\Omega}}\to U_{\partial\Omega}\times \mathbb S^{2n-2}_+,\quad
  (p, v)\mapsto (p,\, \gamma^{-1}_{\nu_p}(v))
\end{equation}
gives a locally $C^{m-1,\,\alpha}$-smooth local trivialization of $S_{\partial\Omega}$ over  $U_{\partial\Omega}$.

\item  [{\rm (ii)}]
Let $\Omega$ and $\widetilde{\Omega}$ be bounded strongly linearly convex domains in $\mathbb C^n,\, n>1$, with $C^3$-smooth boundary, and let $f\!: \Omega\to\widetilde{\Omega}$ be a biholomorphic mapping. If $\varphi$ is a complex geodesic of $\Omega$ with $\varphi^{\ast}$ as its dual mapping, then $f\circ\varphi$ is a complex geodesic of $\widetilde{\Omega}$ whose dual mapping $(f\circ\varphi)^{\ast}$ is related to $\varphi^{\ast}$ via the formula
   $$
   \varphi^{\ast}=(f'\circ\varphi)^t(f\circ\varphi)^{\ast}.
   $$
In particular if $U\in U(n)$ and $\varphi_{p,\, v}$ is a preferred complex geodesic of $\Omega$ associated to some $(p, v)\in S_{\partial\Omega}$, then $U\circ \varphi_{p,\, v}$ is the preferred complex geodesic of $\widetilde{\Omega}:=U(\Omega)$ associated to $(Up,\, Uv)\in S_{\partial\widetilde{\Omega}}$. Roughly speaking, unitary transformations preserve preferred complex geodesics.

\item  [{\rm (iii)}]
Let $\Omega$, $\varphi$ and $\varphi^{\ast}$ be as in ${\rm (ii)}$. Suppose further that $\partial\Omega$ is $C^{3,\,\alpha}$-smooth and $\rho\!:\mathbb C^n \to \mathbb R$ is a $C^{3,\,\alpha}$-defining function for $\Omega$ whose (real) gradient is of length two on $\partial\Omega$. Then
$\langle\rho_z\circ\varphi(1),\, \overline{\varphi'(1)}\rangle>0$ and
\begin{align*}
   \begin{split}
       &\left.\frac{d}{d\theta}\right|_{\theta=0}
         \big|\langle\rho_z\circ\varphi(e^{i\theta}),\,\overline{\varphi'(e^{i\theta})}\rangle\big|^2
        =2\langle\rho_z\circ\varphi(1),\, \overline{\varphi'(1)}\rangle
         \left.\frac{d}{d\theta}\right|_{\theta=0}{\rm Re}
         \langle\rho_z\circ\varphi(e^{i\theta}),\, \overline{\varphi'(e^{i\theta})}\rangle\\
       &\qquad=-2\langle\rho_z\circ\varphi(1),\, \overline{\varphi'(1)}\rangle
        {\rm Im}\Big(\langle \rho_z\circ\varphi(1),\,\overline{\varphi''(1)}\rangle+
        \big\langle\rho_{zz}\circ\varphi(1)\varphi'(1),\, \overline{\varphi'(1)}\big\rangle\Big),
   \end{split}
\end{align*}
where
   $$
   \rho_z=\bigg(\frac{\partial\rho}{\partial z_1},\ldots, \frac{\partial\rho}{\partial z_n}\bigg)^t
   \quad {\rm and} \quad
   \rho_{zz}=\bigg(\frac{\partial^2\rho}{\partial z_j\partial z_k}\bigg)_{1\leq j,\, k\leq n}.
   $$
Observe also that
$|\varphi^{\ast}(\zeta)|=1/\langle\zeta\rho_z\circ\varphi(\zeta),\,\overline{\varphi'(\zeta)}\rangle$ for all $\zeta\in\partial\Delta$. The above calculation shows that
\begin{equation}\label{dual-nor-cond}
  \left.\frac{d}{d\theta}\right|_{\theta=0}|\varphi^{\ast}(e^{i\theta})|=0
\end{equation}
if and only if
   $$
   {\rm Im}\Big(\langle\rho_z\circ\varphi(1),\,\overline{\varphi''(1)}\rangle+
   \big\langle\rho_{zz}\circ\varphi(1)\varphi'(1),\, \overline{\varphi'(1)}\big\rangle\Big)=0.
   $$

\item  [{\rm (iv)}]
Note that for every $m\in\mathbb N$ and $0<\alpha<1$, the norms $\|\ \|_{C^{m,\,\alpha}(\overline{\Delta})}$ and $\|\ \|_{C^{m,\,\alpha}(\partial\Delta)}$ are equivalent on the space $\mathcal{O}(\Delta,\,\mathbb C^n)\cap C^{m,\,\alpha}(\overline{\Delta},\,\mathbb C^n)$. This enables us to identify mappings in $\mathcal{O}(\Delta,\,\mathbb C^n)\cap C^{m,\,\alpha}(\overline{\Delta},\,\mathbb C^n)$ with their trace on $\partial\Delta$. Recall that in Subsection~\ref{Lem:RH-problem} we have denoted by
$\mathcal{O}^{m,\,\alpha}(\partial\Delta,\,\mathbb C^n)$ the space consisting of these elements. With this notation, the conclusion of Theorem \ref{thm:parameter dependence} is equivalent to saying that
    $$
    \Phi\!: S_{\partial\Omega}\to\mathcal{O}^{k,\,\varepsilon}(\partial\Delta,\,\mathbb C^n),
            \quad (p, v)\mapsto \varphi_{p,\, v}
    $$
is locally $C^{m-k-1,\,\alpha-\varepsilon}$-smooth for each $k\in\{2,\ldots,m-2\}$ and $\varepsilon\in (0,\alpha)$.
\end{enumerate}
\end{remark}

We are now ready to prove Theorem \ref{thm:parameter dependence}.

\begin{proof}[Proof of Theorem {\rm\ref{thm:parameter dependence}}]
Fix an arbitrary point $(p_0, v_0)\in S_{\partial\Omega}$. We shall prove that $\Phi$ has the desired regularity near $(p_0, v_0)$.

Let $\varphi_0:=\varphi_{p_0,\, v_0}$ denote the  preferred complex geodesic of $\Omega$ associated to $(p_0, v_0)$ and  $\varphi^{\ast}_0$ denote its dual mapping. Then $\varphi_0, \,\varphi^{\ast}_0\in C^{m-1,\,\alpha}(\overline{\Delta})$. With Remark \ref{rem:observations} (ii) in mind, by making a unitary change of coordinates  we may assume that the first two components $\varphi^{\ast}_{0,\, 1},\, \varphi^{\ast}_{0,\, 2}$ of $\varphi^{\ast}_0$ do not vanish simultaneously on $\overline{\Delta}$. Then a standard argument as in Wolff's approach to the corona
theorem (see \cite{Wolff-corona}), together with the well-known regularity properties of the Cauchy--Green transform on $\Delta$ (see, e.g., \cite[Theorem 1.32]{Vekua_book}), yields two functions $\phi_1,\, \phi_2\in \mathcal{O}(\Delta)\cap C^{m-1,\,\alpha}(\overline{\Delta})$ satisfying
   $$
   \phi_1\varphi^{\ast}_{0,\, 1}+\phi_2\varphi^{\ast}_{0,\, 2}=1
   \quad {\rm on}\ \, \overline{\Delta}.
   $$
As in \cite{Lempert86}, we now consider the mapping $A_0\!: \Delta\to \mathbb C^{n\times n}$ given by
\begin{equation*}\label{defn:flatenning}
A_0=\begin{pmatrix}
    \varphi'_{0,\, 1} & -\varphi^{\ast}_{0,\, 2} & -\phi_1\varphi^{\ast}_{0,\, 3} & \cdots & -\phi_1\varphi^{\ast}_{0,\, n} \\
    \varphi'_{0,\, 2} & \varphi^{\ast}_{0,\, 1} & -\phi_2\varphi^{\ast}_{0,\, 3} & \cdots & -\phi_2\varphi^{\ast}_{0,\, n} \\
    \varphi'_{0,\, 3} & 0 & 1 & \cdots & 0 \\
    \vdots & \vdots & \vdots & \ddots &  \vdots \\
    \varphi'_{0,\, n} & 0 & 0 & \cdots & 1 \\
\end{pmatrix}.
\end{equation*}
Then $A_0\in C^{m-2,\,\alpha}(\overline{\Delta},\,\mathbb C^{n\times n})$ and ${\rm det}\, A_0\neq 0$ on $\overline{\Delta}$. Moreover,
\begin{equation}\label{eq:A_0-flatenning}
   A^t_0\varphi^{\ast}_0=e_1:=(1,0,\ldots, 0)^t.
\end{equation}
We also choose a $C^{m,\,\alpha}$-defining function $\rho\!:\mathbb C^n \to \mathbb R$ for $\Omega$ with gradient  of length two on $\partial\Omega$, so that the unit outward normal vector field of $\partial\Omega$ is exactly
   $$
   \rho_{\bar{z}}:=\bigg(\frac{\partial\rho}{\partial\bar{z}_1},\ldots, \frac{\partial\rho}{\partial\bar{z}_n}\bigg)^t.
   $$
Let $k$ and $\varepsilon$ be as in the theorem. Since the winding number is stable under a small perturbation, Theorem \ref{thm:Lempert1}  implies that a mapping $\varphi\in\mathcal{O}^{k,\,\varepsilon}(\partial\Delta,\,\mathbb C^n)$ near $\varphi_0$ is (the boundary value of) a complex geodesic  of $\Omega$ associated to $(p, v)\in S_{\partial\Omega}$ if and only if it satisfies
\begin{equation*}\label{eq:cg-condition}
\left\{
  \begin{aligned}
    &\,\rho\circ\varphi=0, \\
    &\,\exists\, \mu\in C^{k,\,\varepsilon}(\partial\Delta,\,\mathbb R^+),\,
    \mbox{s.t.}\, \, {\rm Id_{\partial\Delta}}\mu\rho_{z}\circ\varphi\in
    \mathcal{O}^{k,\,\varepsilon}(\partial\Delta,\,\mathbb C^n),\\
    &\,\varphi(1)=p, \\
    &\,\varphi'(1)=\langle v, \nu_p\rangle v.
  \end{aligned}
\right.
\end{equation*}
Thanks to \eqref{eq:A_0-flatenning}, for each such $\varphi$ the second of the above conditions can be equivalently reformulated as
   $$
   [A^t_0\,\rho_{z}\circ\varphi]/(A^t_0\,\rho_{z}\circ\varphi)_1\in \mathcal{O}^{k,\,\varepsilon}(\partial\Delta,\,\mathbb C^{n-1}).
   $$
Here and henceforward, the first component of a vector $z\in \mathbb C^n$ is denoted (as before) by $z_1$ or $(z)_1$, whereas the last $(n-1)$ components (considered as a vector in $\mathbb C^{n-1}$) are denoted by $[z]$.

Let $\mathcal{H}$ denote the standard Hilbert transform on $C^{k,\,\varepsilon}(\partial\Delta,\,\mathbb C^{n-1})$ and define a self-mapping $\Pi$ of $C^{k,\,\varepsilon}(\partial\Delta,\,\mathbb C^{n-1})$ by
  $$
  \Pi u=\frac12(u-i\mathcal{H}u)-\frac{1}{4\pi}\int_{0}^{2\pi}u(e^{i\theta})d\theta.
  $$
Then $\Pi$ is continuous and idempotent (i.e., $\Pi^2=\Pi$), and
${\rm Ker}\,\Pi=\mathcal{O}^{k,\,\varepsilon}(\partial\Delta,\,\mathbb C^{n-1})$. This leads us to consider the mapping
\begin{equation*}
\begin{split}
  \Theta\!:\big(&\mathcal{O}^{k,\,\varepsilon}(\partial\Delta,\,\mathbb C^n)\cap B_{k,\,
  \varepsilon} (\varphi_0,\,\varepsilon_0)\big)\times\partial\Omega\times\mathbb B^{n-1}\to\\
  &C^{k,\,\varepsilon}(\partial\Delta,\,\mathbb R)\times
  \Pi\big(C^{k,\,\varepsilon}(\partial\Delta,\,\mathbb C^{n-1})\big)\times\mathbb
  C^n\times\mathbb C^n\times \mathbb R
\end{split}
\end{equation*}
given by
\begin{equation*}\label{defn:Theta}
\begin{split}
  \Theta(\varphi,\, p,\, \hat{v})=\bigg(&\rho\circ \varphi,\,
  \Pi\bigg(\frac{[A^t_0\,\rho_{z}\circ\varphi]}{(A^t_0\,\rho_{z}\circ\varphi)_1}\bigg),\,
  \varphi(1)-p,\, \varphi'(1)-\sqrt{1-|\hat{v}|^2}\,\gamma_{\nu_p}\big(\sqrt{1-|\hat{v}|^2},\,
  \hat{v}^t\big)^t,\, \\
  &{\rm Im}\Big(\langle\rho_{z}\circ\varphi(1),\,\overline{\varphi''(1)}\rangle+
   \big\langle\rho_{zz}\circ\varphi(1)\varphi'(1),\, \overline{\varphi'(1)}\big\rangle\Big)\bigg),
\end{split}
\end{equation*}
where $B_{k,\, \varepsilon} (\varphi_0,\,\varepsilon_0)$ denotes the open ball in
$C^{k,\,\varepsilon}(\partial\Delta,\,\mathbb C^n)$ with center $\varphi_0$ and radius $\varepsilon_0>0$, which is a sufficiently small constant. By virtue of the preceding argument and Remark \ref{rem:observations} (iii), we conclude that a mapping $\varphi\in\mathcal{O}^{k,\,\varepsilon}(\partial\Delta,\,\mathbb C^n)$ near $\varphi_0$ is the preferred complex geodesic of $\Omega$ associated to some $(p, v)\in S_{\partial\Omega}$ near $(p_0, v_0)$ if and only if it solves the equation
\begin{equation}\label{eq:CG-equation}
  \Theta(\varphi,\, p,\, \hat{v})=0.
\end{equation}
Here, with a slight abuse of notation we identify $(p,\, \hat{v})$ with
  $$
 (p, v):=\big(p,\, \gamma_{\nu_p}\big(\sqrt{1-|\hat{v}|^2},\, \hat{v}^t\big)^t\big)
  $$
via the local trivialization \eqref{defn:local-trivialization} of $S_{\partial\Omega}$ over  $U_{\partial\Omega}\ni p_0$. Therefore to investigate the regularity of the mapping  $\Phi$ (in the theorem) near $(p_0, v_0)$, it is sufficient to solve equation \eqref{eq:CG-equation} with
$(p, \hat{v})\in\partial\Omega\times\mathbb B^{n-1}$ given near $(p_0, \hat{v}_0)$.

Now it is routine to verify that $\Theta$ is $C^{m-k-1,\,\alpha-\varepsilon}$-smooth provided $\varepsilon_0>0$ is sufficiently small. An application of the implicit function theorem in Banach spaces would yield the desired regularity of $\Phi$, if we could prove that the partial derivative
\begin{equation*}
L:=\big(d\Theta(\,\cdot\,,\, p_0,\, \hat{v}_0)\big)_{\varphi_0}\!:
\mathcal{O}^{k,\,\varepsilon}(\partial\Delta,\,\mathbb C^n)\to C^{k,\,\varepsilon}(\partial\Delta,\,\mathbb R)\times\Pi\big(C^{k,\,\varepsilon}(\partial\Delta,\,\mathbb C^{n-1})\big)
\times\mathbb C^n\times\mathbb C^n\times \mathbb R
\end{equation*}
is invertible. An easy calculation using \eqref{eq:A_0-flatenning} gives
\begin{align*}
 \begin{split}
   L(\varphi)&=\left.\frac{d}{d\tau}\right|_{\tau=0}\Theta(\varphi_0+\tau\varphi,\, p_0,\, \hat{v}_0)\\
     &=\bigg(2{\rm Re}\langle \varphi,\, \rho_{\bar{z}}\circ\varphi_0\rangle,\,
       \Pi\bigg(\frac{[A^t_0(\rho_{zz}\circ\varphi_0)\varphi+A^t_0(\rho_{z\bar{z}}\circ\varphi_0)\overline{\varphi}]}
       {(A^t_0\,\rho_{z}\circ\varphi_0)_1}\bigg),\,
   \varphi(1),\, \varphi'(1),\, \hat{L}(\varphi)\bigg),
\end{split}
\end{align*}
where $\varphi=(\varphi_1,\ldots,\varphi_n)^t\in \mathcal{O}^{k,\,\varepsilon}(\partial\Delta,\,\mathbb C^n)$ and
\begin{align}\label{eq:partial-der}
\begin{split}
\hat{L}(\varphi)={\rm Im}\bigg(&\langle\rho_z(p_0),\, \overline{\varphi''(1)}\rangle
       +2\big\langle\rho_{zz}(p_0)\varphi'_0(1),\, \overline{\varphi'(1)}\big\rangle\\
      &+\big\langle\rho_{zz}(p_0)\varphi(1)+\rho_{z\bar{z}}(p_0)\overline{\varphi(1)},\,
        \overline{\varphi''_0(1)}\big\rangle\\
      &+\Big\langle\Big(\sum_{j=1}^{n}\frac{\partial \rho_{zz}}{\partial z_j}(p_0)\varphi_j(1)
       +\sum_{j=1}^{n}\frac{\partial \rho_{zz}}{\partial \bar{z}_j}(p_0)\overline{\varphi_j(1)}
           \Big)\varphi'_0(1),\, \overline{\varphi'_0(1)}\Big\rangle\bigg).
\end{split}
\end{align}
Thus the invertibility of $L$ means that for every
   $$
   (r,\, f,\, w,\,  v,\, \lambda)\in C^{k,\,\varepsilon}(\partial\Delta,\,\mathbb R)\times
   \Pi\big(C^{k,\,\varepsilon}(\partial\Delta,\,\mathbb C^{n-1})\big)\times\mathbb C^n
   \times\mathbb C^n\times \mathbb R,
   $$
the following system of equations admits exactly one solution
$\varphi\in \mathcal{O}^{k,\,\varepsilon}(\partial\Delta,\,\mathbb C^n)$:
\begin{equation*}\label{eq:L-invertibility1}
\left\{
  \begin{aligned}
    &\,{\rm Re}\langle \varphi,\, \rho_{\bar{z}}\circ\varphi_0\rangle=r,\\
    &\,\Pi\bigg(\frac{[A^t_0(\rho_{zz}\circ\varphi_0)\varphi+A^t_0(\rho_{z\bar{z}}\circ\varphi_0)\overline{\varphi}]}
    {(A^t_0\,\rho_{z}\circ\varphi_0)_1}\bigg)=f,\\
    &\,\varphi(1)=w,\\
    &\,\varphi'(1)=v,\\
    &\,\hat{L}(\varphi)=\lambda.
  \end{aligned}
\right.
\end{equation*}

To solve this system, we use Theorem \ref{thm:RH-prob1} and the basic properties of the Hilbert transform $\mathcal{H}$. Writing
   $$
   \left.\varphi^{\ast}_0\right|_{\partial\Delta}={\rm Id_{\partial\Delta}}\mu_0\,\rho_{z}\circ\varphi_0
   $$
and using \eqref{eq:A_0-flatenning}, we see that $\varphi$ solves the above system if  and only if $\psi=A^{-1}_0\varphi$ satisfies the following conditions:
\begin{equation}\label{eq:L-invertibility21}
   {\rm Re}\,(\psi_1/{\rm Id_{\partial\Delta}})=\mu_0r,
\end{equation}
\begin{equation}\label{eq:L-invertibility22}
   \Pi\bigg({\rm Id_{\partial\Delta}}\mu_0\big[A^t_0(\rho_{zz}\circ\varphi_0)A_0\psi
            +A^t_0(\rho_{z\bar{z}}\circ\varphi_0)\overline{A_0}\,\overline{\psi}\,\big]\bigg)=f,
\end{equation}
\begin{equation}\label{eq:L-invertibility23}
   \psi(1)=A^{-1}_0(1)w,
\end{equation}
\begin{equation}\label{eq:L-invertibility24}
   \psi'(1)=A^{-1}_0(1)v+(A^{-1}_0)'(1)w,
\end{equation}
and
\begin{equation}\label{eq:L-invertibility25}
   \hat{L}(A_0\psi)=\lambda.
\end{equation}
We first show that $\psi_1$ is uniquely determined by conditions \eqref{eq:L-invertibility21} and \eqref{eq:L-invertibility23}-\eqref{eq:L-invertibility25}. For this we rephrase condition \eqref{eq:L-invertibility21} as
   $$
   {\rm Re}\bigg(\frac{\psi_1-\psi_1(0)}{{\rm Id_{\partial\Delta}}}+
   \overline{\psi_1(0)}\,{\rm Id_{\partial\Delta}}\bigg)=\mu_0r.
   $$
Set
   $$
   f_0:=\mu_0r+i\mathcal{H}(\mu_0r).
   $$
Then by the Privalov theorem $f_0\in \mathcal{O}^{k,\,\varepsilon}(\partial\Delta,\,\mathbb C)$, and
${\rm Im}\,f_0(0)=0$ by the definition of $\mathcal{H}$. It follows that
   $$
   \frac{\psi_1-\psi_1(0)}{{\rm Id_{\partial\Delta}}}+
   \overline{\psi_1(0)}\,{\rm Id_{\partial\Delta}}=f_0+i{\rm Im}\,\psi'_1(0),
   $$
i.e.,
\begin{equation}\label{eq:formula-for-psi_1}
\psi_1=\psi_1(0)+{{\rm Id_{\partial\Delta}}}\big(f_0-\overline{\psi_1(0)}\,{\rm Id_{\partial\Delta}}
       +i{\rm Im}\,\psi'_1(0)\big).
\end{equation}
We then need to show that $\psi_1(0)$ and ${\rm Im}\,\psi'_1(0)$ are uniquely determined by conditions \eqref{eq:L-invertibility23}-\eqref{eq:L-invertibility25}. To this end, we first observe that  \eqref{eq:L-invertibility23} and \eqref{eq:L-invertibility24} uniquely determine $\psi(1)$ and $\psi'(1)$, which in turn determine
${\rm Im}\langle\varphi''(1),\, \nu_{p_0}\rangle$ (via \eqref{eq:L-invertibility25} and \eqref{eq:partial-der}). On the other hand,  \eqref{eq:formula-for-psi_1} gives
\begin{equation*}\label{eq:psi_1-determination}
\left\{
 \begin{aligned}
     &\,2{\rm Im}\,\psi_1(0)+{\rm Im}\,\psi'_1(0)={\rm Im}\big(\psi_1(1)-f_0(1)\big),\\
     &\,{\rm Re}\,\psi_1(0)=\frac12{\rm Re}\big(f_0(1)+f'_0(1)-\psi'_1(1)\big),\\
     &\,{\rm Im}\,\psi_1(0)=\frac12{\rm Re}\big(\psi''_1(1)-f''_0(1)-2f'_0(1)\big),
 \end{aligned}
\right.
\end{equation*}
indicating that $\psi_1(0)$ and ${\rm Im}\,\psi'_1(0)$ are expressible in terms of $\psi_1(1),\, \psi_1'(1)$ and ${\rm Im}\,\psi''_1(1)$ in a unique way. Thus the problem reduces to proving that  ${\rm Im}\,\psi''_1(1)$ is uniquely determined by ${\rm Im}\langle\varphi''(1),\, \nu_{p_0}\rangle$. But this is clearly true, as can be easily seen from \eqref{eq:A_0-flatenning} and the fact that $A_0\in \mathcal{O}^{2,\,\alpha}(\partial\Delta,\,\mathbb C^{n\times n})$ and ${\rm det}\, A_0\neq 0$ on $\overline{\Delta}$.

Up to now we have proved that $\psi_1$ is uniquely determined by conditions \eqref{eq:L-invertibility21} and \eqref{eq:L-invertibility23}-\eqref{eq:L-invertibility25}. We next consider the last $(n-1)$ components $[\psi]=(\psi_2,\ldots,\psi_n)^t$. Set
   $$
   H_0=(h_{ij})_{1\leq i,\,j\leq n}:=\mu_0A^t_0(\rho_{z\bar{z}}\circ\varphi_0)\overline{A_0},
   $$
   $$
   S_0=(s_{ij})_{1\leq i,\,j\leq n}:=({\rm Id_{\partial\Delta}})^2\mu_0A^t_0(\rho_{zz}\circ\varphi_0)A_0,
   $$
and
\begin{equation*}\label{defn:Her-sym}
H:=(h_{ij})_{2\leq i,\,j\leq n},\quad S:=(s_{ij})_{2\leq i,\,j\leq n}.
\end{equation*}
Also we put $g:=[\psi]$. Then condition \eqref{eq:L-invertibility22} amounts to
\begin{equation}\label{eq:RH-problem}
   H\overline{(g/{\rm Id_{\partial\Delta}})}+Sg/{\rm Id_{\partial\Delta}}+\widetilde{f}
   \in\mathcal{O}^{k,\,\varepsilon}(\partial\Delta,\,\mathbb C^{n-1}),
\end{equation}
where
   $$
   \widetilde{f}:=-f+\overline{(\psi_1/{\rm Id_{\partial\Delta}})}(h_{21},\ldots, h_{n1})^t+
                  (\psi_1/{\rm Id_{\partial\Delta}})(s_{21},\ldots, s_{n1})^t
                 \in C^{k,\,\varepsilon}(\partial\Delta,\,\mathbb C^{n-1}).
   $$
Thus it suffices to show that there exists a unique $g\in\mathcal{O}^{k,\,\varepsilon}(\partial\Delta,\,\mathbb C^{n-1})$ with $g(1)=[\psi(1)]$ and $g'(1)=[\psi'(1)]$ such that \eqref{eq:RH-problem} holds. Now since $\Omega$ is strongly linearly convex, it follows that
   $$
   \sum_{i,\, j=1}^n\frac{\partial^2\rho}{\partial z_i\partial\bar {z}_j}\circ\varphi_0(\zeta)v_i\bar{v}_j
   >\bigg|\sum_{i,\, j=1}^n\frac{\partial^2\rho}{\partial z_i\partial z_j}\circ\varphi_0(\zeta)v_iv_j\bigg|
   $$
for all $\zeta\in\partial\Delta$ and non-zero $v\in T_{\varphi_0(\zeta)}^{1,\, 0}\partial\Omega$. This means precisely that
   $$
   \hat{v}^tH(\zeta)\bar{\hat{v}}>|\hat{v}^tS(\zeta)\hat{v}|
   $$
for all $\zeta\in\partial\Delta$ and $\hat{v}\in \mathbb C^{n-1}\!\setminus\!\{0\}$. Theorem \ref{thm:RH-prob1} then guarantees the existence of a unique  $g\in\mathcal{O}^{k,\,\varepsilon}(\partial\Delta,\,\mathbb C^{n-1})$ with the desired property. This proves the invertibility of $L$ and hence the theorem.
\end{proof}

\section{Proof of Theorem \ref{thm:joint-regularity}} \label{sect:proof of thm B}

Recall that the boundary of the domain $\Omega$ in the theorem is assumed to be $C^{m,\,\alpha}$-smooth, where $m\geq 4$ is an integer and $0<\alpha<1$.
Depending on the value of $m$ we argue in two cases.

\medskip
\noindent  {\bf Case 1:} $m=4$.
\smallskip

By Theorem \ref{thm:parameter dependence} (with $k=2$) we know that
   $$
   \Phi\!: S_{\partial\Omega}\to C^{2,\,\varepsilon}(\overline{\Delta},\,\mathbb C^n),
           \quad (p, v)\mapsto \varphi_{p,\, v}
   $$
is locally $C^{1,\,\alpha-\varepsilon}$-smooth for all $0<\varepsilon<\alpha$. Note also that the evaluation mapping
   $$
   e\!: C^{2,\,\varepsilon}(\overline{\Delta},\,\mathbb C^n)\times \overline{\Delta}\to\mathbb C^n,\quad
       (\varphi,\, \zeta) \mapsto \varphi(\zeta)
   $$
is locally $C^{2,\,\varepsilon}$-smooth. It then follows that
$\widetilde{\Phi}\!=e\circ(\Phi\times {\rm Id}_{\overline{\Delta}})$
is necessarily locally $C^{1,\,\alpha-\varepsilon}$-smooth for all $0<\varepsilon<\alpha$.

\medskip
\noindent  {\bf Case 2:} $m\geq 5$.
\smallskip

Let
   $$
   U_{S_{\partial\Omega}}\to U_{\partial\Omega}\times \mathbb S^{2n-2}_+,\quad
   (p, v)\mapsto (p,\, \gamma^{-1}_{\nu_p}(v))
   $$
be a locally $C^{m-1,\,\alpha}$-smooth local trivialization of $S_{\partial\Omega}$ over  an open subset $U_{\partial\Omega}$ of $\partial\Omega$, which is as in \eqref{defn:local-trivialization}. For every
$(p, \hat{v})\in U_{\partial\Omega}\times \mathbb B^{n-1}$ we write
   $$
   \varphi_{p,\,\hat{v}}:=\varphi_{p,\,\gamma_{\nu_p}(v)}\quad
    {\rm with}\quad v:=(\sqrt{1-|\hat{v}|^2},\, \hat{v}^t)^t
   $$
to simplify the notation. Then what we need to prove is no other than that for every $r\in (0, 1)$ the  mapping
\begin{equation}\label{defn:variation-local-version}
   \widetilde{\Phi}_r\!: U_{\partial\Omega}\times r\mathbb B^{n-1}\times\overline{\Delta}\to \mathbb C^n,
   \quad (p, \hat{v}, \zeta)\mapsto \varphi_{p,\,\hat{v}}(\zeta)
\end{equation}
is $C^{m-3,\,\alpha-\varepsilon}$-smooth for all $0<\varepsilon<\alpha$, after shrinking $U_{\partial\Omega}$ if necessary.

First of all, a similar argument as in the proof of \cite[Theorem 3.1]{Huang-Wang} shows
   $$
   \big\{\varphi_{p,\,\hat{v}}(0)\!: (p, \hat{v})\in  U_{\partial\Omega}\times
   r\mathbb B^{n-1}\big\}\subset\subset\Omega,
   $$
whence
   $$
   \sup\left\{\|\varphi_{p,\,\hat{v}}\|_{C^{m-1,\,\min\{\alpha,\, 1/2\}}(\overline{\Delta})}\!:
       (p, \hat{v})\in  U_{\partial\Omega}\times r\mathbb B^{n-1}\right\}<\infty
   $$
for all $r\in(0, 1)$ (see \cite[Proposition 1]{Huang-Illinois94}).
Combining this with Theorem \ref{thm:parameter dependence} then yields
\begin{align}\label{ineq:Holder-estimate1}
\begin{split}
\big|\varphi^{(k)}_{p_1,\,\hat{v}_1}(\zeta_1)-\varphi^{(k)}_{p_2,\,\hat{v}_2}(\zeta_2)\big|
&\leq\big|\varphi^{(k)}_{p_1,\,\hat{v}_1}(\zeta_1)-\varphi^{(k)}_{p_1,\,\hat{v}_1}(\zeta_2)\big|
    +\big|\varphi^{(k)}_{p_1,\,\hat{v}_1}(\zeta_2)-\varphi^{(k)}_{p_2,\,\hat{v}_2}(\zeta_2)\big|\\
&\leq\|\varphi_{p_1,\,\hat{v}_1}\|_{C^{m-1}(\overline{\Delta})}|\zeta_1-\zeta_2|
    +\|\varphi_{p_1,\,\hat{v}_1}-\varphi_{p_2,\,\hat{v}_2}\|_{C^{m-2}(\overline{\Delta})}\\
&\leq {\rm const}_{r,\, \Omega}|(p_1, \hat{v}_1, \zeta_1)-(p_2, \hat{v}_2, \zeta_2)|
\end{split}
\end{align}
for $k=1,\ldots, m-2$, and $(p_1, \hat{v}_1, \zeta_1),\, (p_2, \hat{v}_2, \zeta_2)\in U_{\partial\Omega}\times r\mathbb B^{n-1}\times\overline{\Delta}$.

Before proceeding with further estimates, let us write
   $$
   \Phi(p, \hat{v}):=\Phi\big(p,\, \gamma_{\nu_p}\big(\sqrt{1-|\hat{v}|^2},\, \hat{v}^t\big)^t\big)
   $$
for all $(p, \hat{v})\in U_{\partial\Omega}\times \mathbb B^{n-1}$ and identify $p$ with its local coordinates via a coordinate chart. Consider all multi-indices $(\lambda, \mu)\in \mathbb N^{2n-1}\times\mathbb N^{2n-2}$ with
   $$
   1\leq|\lambda|+|\mu|\leq m-3.
   $$
When $|\lambda|+|\mu|=m-3$, we conclude from Theorem \ref{thm:parameter dependence} (with $k=2$) that
   $$
   \frac{\partial^{m-3}\Phi}{\partial p^{\lambda}\partial \hat{v}^{\mu}}(p, \hat{v})
   \in\mathcal{O}(\Delta,\,\mathbb C^n)\cap C^{2,\,\varepsilon}(\overline{\Delta},\,\mathbb C^n),\quad \forall\,(p, \hat{v})\in U_{\partial\Omega}\times \mathbb B^{n-1}.
   $$
Moreover,
\begin{align}\label{ineq:Holder-estimate2}
 \begin{split}
  &\bigg|\bigg(\frac{\partial^{m-3}\Phi}{\partial p^{\lambda}\partial \hat{v}^{\mu}}
                (p_1, \hat{v}_1)\bigg)(\zeta_1)
               -\bigg(\frac{\partial^{m-3}\Phi}{\partial p^{\lambda}\partial \hat{v}^{\mu}}
                (p_2, \hat{v}_2)\bigg)(\zeta_2)
          \bigg|\\
  &\leq \bigg|\bigg(\frac{\partial^{m-3}\Phi}{\partial p^{\lambda}\partial \hat{v}^{\mu}}
              (p_1, \hat{v}_1)\bigg)(\zeta_1)
       -\bigg(\frac{\partial^{m-3}\Phi}{\partial p^{\lambda}\partial \hat{v}^{\mu}}
              (p_1, \hat{v}_1)\bigg)(\zeta_2)
        \bigg|\\
  &\quad\, +\bigg|\bigg(\frac{\partial^{m-3}\Phi}{\partial p^{\lambda}\partial \hat{v}^{\mu}}
              (p_1, \hat{v}_1)-\frac{\partial^{m-3}\Phi}{\partial p^{\lambda}\partial \hat{v}^{\mu}}
              (p_2, \hat{v}_2)\bigg)(\zeta_2)\bigg|\\
  &\leq \bigg\|\frac{\partial^{m-3}\Phi}{\partial p^{\lambda}\partial \hat{v}^{\mu}}
              (p_1,\hat{v}_1)\bigg\|_{C^{1}(\overline{\Delta})}|\zeta_1-\zeta_2|
           +\bigg\|\frac{\partial^{m-3}\Phi}{\partial p^{\lambda}\partial \hat{v}^{\mu}}
              (p_1, \hat{v}_1)-\frac{\partial^{m-3}\Phi}{\partial p^{\lambda}\partial \hat{v}^{\mu}}
              (p_2, \hat{v}_2)\bigg\|_{C^{0}(\overline{\Delta})}\\
  &\leq {\rm const}_{r,\,\alpha,\,\varepsilon,\,\Omega}
   \big(|\zeta_1-\zeta_2|+|(p_1, \hat{v}_1)-(p_2, \hat{v}_2)|^{\alpha-\varepsilon}\big)
 \end{split}
\end{align}
for all $(p_1, \hat{v}_1, \zeta_1),\, (p_2, \hat{v}_2, \zeta_2)\in U_{\partial\Omega}\times r\mathbb B^{n-1}\times\overline{\Delta}$. We now turn to the case when $1\leq |\lambda|+|\mu|<m-3$. Note that
   $$
   2\leq m-|\lambda|-|\mu|-2\leq m-3.
   $$
Using Theorem \ref{thm:parameter dependence} with $k:=m-|\lambda|-|\mu|-2$ we obtain
   $$
   \frac{\partial^{|\lambda|+|\mu|}\Phi}{\partial p^{\lambda}\partial \hat{v}^{\mu}}(p, \hat{v})
   \in\mathcal{O}(\Delta,\,\mathbb C^n)\cap C^{m-|\lambda|-|\mu|-2,\,\varepsilon}(\overline{\Delta},\,\mathbb C^n), \quad \forall\,(p, \hat{v})\in U_{\partial\Omega}\times \mathbb B^{n-1},
   $$
and
\begin{align}\label{ineq:Holder-estimate3}
 \begin{split}
  \bigg|&\frac{\partial^{j+|\lambda|+|\mu|}\widetilde{\Phi}_r}{\partial \zeta^j\partial p^{\lambda}
                \partial \hat{v}^{\mu}}(p_1, \hat{v}_1, \zeta_1)
                -\frac{\partial^{j+|\lambda|+|\mu|}\widetilde{\Phi}_r}{\partial \zeta^j\partial p^{\lambda}
                \partial \hat{v}^{\mu}}(p_2, \hat{v}_2, \zeta_2)
          \bigg|\\
  &=\bigg|\frac{\partial^j}{\partial \zeta^j}\bigg(\frac{\partial^{|\lambda|+|\mu|}\Phi}
          {\partial p^{\lambda}\partial \hat{v}^{\mu}}(p_1, \hat{v}_1)\bigg)(\zeta_1)
          -\frac{\partial^j}{\partial \zeta^j}\bigg(\frac{\partial^{|\lambda|+|\mu|}\Phi}
          {\partial p^{\lambda}\partial \hat{v}^{\mu}}(p_2, \hat{v}_2)\bigg)(\zeta_2)\bigg|\\
  &\leq\bigg|\frac{\partial^j}{\partial \zeta^j}\bigg(\frac{\partial^{|\lambda|+|\mu|}\Phi}
            {\partial p^{\lambda}\partial \hat{v}^{\mu}}(p_1, \hat{v}_1)\bigg)(\zeta_1)
            -\frac{\partial^j}{\partial \zeta^j}\bigg(\frac{\partial^{|\lambda|+|\mu|}\Phi}
            {\partial p^{\lambda}\partial \hat{v}^{\mu}}(p_1, \hat{v}_1)\bigg)(\zeta_2)\bigg|\\
  &\quad\, +\bigg|\frac{\partial^j}{\partial \zeta^j}
            \bigg(\frac{\partial^{|\lambda|+|\mu|}\Phi}{\partial p^{\lambda}\partial \hat{v}^{\mu}}
                  (p_1, \hat{v}_1)
                 -\frac{\partial^{|\lambda|+|\mu|}\Phi}{\partial p^{\lambda}\partial \hat{v}^{\mu}}
                  (p_2, \hat{v}_2)\bigg)(\zeta_2)
            \bigg|\\
  &\leq\bigg\|\frac{\partial^{|\lambda|+|\mu|}\Phi}{\partial p^{\lambda}\partial \hat{v}^{\mu}}
              (p_1,\hat{v}_1)\bigg\|_{C^{m-|\lambda|-|\mu|-3}(\overline{\Delta})}|\zeta_1-\zeta_2|\\
  &\quad\, +\bigg\|\frac{\partial^{|\lambda|+|\mu|}\Phi}{\partial p^{\lambda}\partial \hat{v}^{\mu}}
              (p_1, \hat{v}_1)-\frac{\partial^{|\lambda|+|\mu|}\Phi}
              {\partial p^{\lambda}\partial \hat{v}^{\mu}}(p_2, \hat{v}_2)
       \bigg\|_{C^{m-|\lambda|-|\mu|-4}(\overline{\Delta})}\\
  &\leq{\rm const}_{r,\, \Omega}|(p_1, \hat{v}_1, \zeta_1)-(p_2, \hat{v}_2, \zeta_2)|
 \end{split}
\end{align}
for all $(p_1, \hat{v}_1, \zeta_1),\, (p_2, \hat{v}_2, \zeta_2)\in U_{\partial\Omega}\times r\mathbb B^{n-1}\times\overline{\Delta}$ and all $j\in \mathbb N$ with
   $$
  j<m-|\lambda|-|\mu|-3.
   $$
Similarly, for $j=m-|\lambda|-|\mu|-3$ an application of  Theorem \ref{thm:parameter dependence} with
   $$
  k:=m-|\lambda|-|\mu|-1
   $$
yields
\begin{align}\label{ineq:Holder-estimate4}
 \begin{split}
   \bigg|&\frac{\partial^{m-3}\widetilde{\Phi}_r}{\partial \zeta^{m-|\lambda|-|\mu|-3}
                 \partial p^{\lambda}\partial \hat{v}^{\mu}}(p_1, \hat{v}_1, \zeta_1)
                -\frac{\partial^{m-3}\widetilde{\Phi}_r}{\partial \zeta^{m-|\lambda|-|\mu|-3}
                 \partial p^{\lambda}\partial \hat{v}^{\mu}}(p_2, \hat{v}_2, \zeta_2)
   \bigg|\\
   &\leq\bigg\|\frac{\partial^{|\lambda|+|\mu|}\Phi}{\partial p^{\lambda}\partial \hat{v}^{\mu}}
              (p_1,\hat{v}_1)\bigg\|_{C^{m-|\lambda|-|\mu|-2}(\overline{\Delta})}|\zeta_1-\zeta_2|\\
   &\quad\, +\bigg\|\frac{\partial^{|\lambda|+|\mu|}\Phi}{\partial p^{\lambda}\partial \hat{v}^{\mu}}
              (p_1, \hat{v}_1)-\frac{\partial^{|\lambda|+|\mu|}\Phi}
              {\partial p^{\lambda}\partial \hat{v}^{\mu}}(p_2, \hat{v}_2)
        \bigg\|_{C^{m-|\lambda|-|\mu|-3}(\overline{\Delta})}\\
  &\leq {\rm const}_{r,\,\alpha,\,\varepsilon,\,\Omega}
   \big(|\zeta_1-\zeta_2|+|(p_1, \hat{v}_1)-(p_2, \hat{v}_2)|^{\alpha-\varepsilon}\big)
 \end{split}
\end{align}
for all $(p_1, \hat{v}_1, \zeta_1),\, (p_2, \hat{v}_2, \zeta_2)\in U_{\partial\Omega}\times r\mathbb B^{n-1}\times\overline{\Delta}$.

Now putting \eqref{ineq:Holder-estimate1}-\eqref{ineq:Holder-estimate4} together, we conclude that for every $r\in (0, 1)$ the mapping $\widetilde{\Phi}_r$ in \eqref{defn:variation-local-version} is indeed  $C^{m-3,\,\alpha-\varepsilon}$-smooth for all $0<\varepsilon<\alpha$, as desired. This completes the proof of the theorem.

\section{Proof of Theorem \ref{thm:regularity-BSR}}\label{sect:proof of thm C}
In light of Theorem \ref{thm:joint-regularity} together with \eqref{defn:inv-BSR} and \eqref{ball-splitting}, it is evident that each $\Psi_p^{-1}$ has the desired regularity.  To establish the same regularity for $\Psi$, by definition it suffices to show that for every sufficiently small $\delta>0$ the mapping
  $$
  (\overline{\Omega}\times\partial\Omega)\!\setminus\!{\rm dist}_{\delta}\, (\overline{\Omega}\times\partial\Omega)
  \to \mathbb C^n\times (\overline{\Delta}\!\setminus\!\{1\}),\quad (z,\, p)\mapsto (v_{z,\,p},\,\zeta_{z,\,p})
  $$
determined in \eqref{defn:BSR} is $C^{m-3,\,\alpha-\varepsilon}$-smooth for all $0<\varepsilon<\alpha$. Observe also that for every
$(z, p)\in\overline{\Omega}\times\partial\Omega$ with $z\neq p$, $(v_{z,\,p},\,\zeta_{z,\,p})$ is the only solution to the equation
  $$
  \widetilde{\Phi}(p, v, \zeta)-z=0.
  $$
It therefore suffices to verify the injectivity of the (real) differential
   $$
   (d\widetilde{\Phi}(p,\,\cdot\,,\,\cdot\,))_{(v,\,\zeta)}\!: \mathbb R^{2n-2}\times\mathbb R^2
   \to \mathbb R^{2n}
   $$
for every $(p, v,\zeta)\in S_{\partial\Omega}\times(\overline{\Delta}\!\setminus\!\{1\})$, for then the desired regularity follows immediately from Theorem \ref{thm:joint-regularity} and  the implicit function theorem.

Fix an arbitrary point $p_0\in \partial\Omega$. Up to a unitary transformation of $\mathbb C^n$ which preserves the strong linear convexity of $\Omega$ and preferred complex geodesics (see Remark \ref{rem:observations} (ii)), we may assume that $\nu_{p_0}=e_1$. Set $\varphi_{\hat{v}}:=\varphi_{p_0,\,v}$ for $v:=(\sqrt{1-|\hat{v}|^2},\, \hat{v}^t)^t$ and consider the mapping
   $$
   \hat{\Phi}\!: \mathbb B^{n-1}\times(\overline{\Delta}\!\setminus\!\{1\})
   \to\overline{\Omega}\!\setminus\!\{p_0\}
   $$
given by $\hat{\Phi}(\hat{v},\, \zeta)=\varphi_{\hat{v}}(\zeta)$. Then by the symmetry of $\mathbb B^{n-1}$,  to show the injectivity of $d\widetilde{\Phi}(p_0,\,\cdot\,,\,\cdot\,)$ we only need to verify
\begin{equation*}\label{eq:trival-kernel}
   {\rm Ker}\,d\hat{\Phi}_{(0,\,\zeta)}=\{0\}, \quad \forall\,\zeta\in\overline{\Delta}\!\setminus\!\{1\}.
\end{equation*}
With the help of (the uniqueness part of) Theorem \ref{thm:RH-prob2}, we shall prove this by first performing for $\Omega$ a holomorphic coordinate transformation as given in Theorem \ref{thmA:normalization} and then carefully analyzing the behavior of the normalization condition \eqref{dual-nor-cond} with respect to such a transformation.

Fix $\zeta_0\in\overline{\Delta}\!\setminus\!\{1\}$ and suppose $(\hat{v},\, \lambda)\in \mathbb C^{n-1}\times\mathbb C$ satisfies
\begin{equation}\label{vanishing}
   d\hat{\Phi}_{(0,\,\zeta_0)}(\hat{v},\,\lambda)=0.
\end{equation}
To show that $\hat{v}=0$ and $\lambda=0$, we note that
$\hat{\Phi}(0,\,\cdot\,)=\varphi_0:=\varphi_{p_0,\, e_1}$, hence
\begin{equation*}\label{eq:horizontal-var}
   d\hat{\Phi}_{(0,\,\zeta_0)}(0, 1)=(d\hat{\Phi}(0,\,\cdot\,))_{\zeta_0}(1)=\varphi'_0(\zeta_0).
\end{equation*}
Setting
\begin{equation*}\label{defn:v-variation}
   \varphi:=\left.\frac{d}{d\tau}\right|_{\tau=0}\varphi_{\tau\hat{v}}\in\mathcal{O}(\Delta,\,\mathbb C^n)
   \cap C^2(\overline{\Delta},\,\mathbb C^n),
\end{equation*}
which makes sense in view of Theorem \ref{thm:parameter dependence}, we obtain
\begin{equation*}\label{eq:vertical-var}
d\hat{\Phi}_{(0,\,\zeta_0)}(\hat{v},\,0)=(d\hat{\Phi}(\,\cdot\,,\zeta_0))_0(\hat{v})=\varphi(\zeta_0).
\end{equation*}
Consequently, \eqref{vanishing} can be rephrased as
\begin{equation}\label{equi-vanishing}
   \varphi(\zeta_0)+\lambda\varphi'_0(\zeta_0)=0.
\end{equation}

To proceed further, we make use of Theorem \ref{thmA:normalization} in Appendix \ref{supplement}. Choose a biholomorphism $F\!:D\to\Omega$ which extends to a $C^2$-smooth diffeomorphism between $\overline{D}$ and $\overline{\Omega}$, and a $C^2$-defining function $\rho$ for $\Omega$ defined on a neighborhood of $\varphi_0(\overline{\Delta})$ in $\mathbb C^n$  such that
\begin{equation}\label{Straightening1}
   F(\,\cdot\,,0)=\varphi_0 \quad {\rm on}\ \, \overline{\Delta},
\end{equation}
\begin{equation}\label{Straightening2}
{\rm Id_{\partial\Delta}}A^t_0\,\rho_{z}\circ\varphi_0=e_1 \quad {\rm on}\ \, \partial\Delta,
\end{equation}
and
\begin{equation}\label{Straightening3}
A^t_0(\rho_{z\bar{z}}\circ\varphi_0)\overline{A_0}=I_n,
   \quad A^t_0(\rho_{zz}\circ\varphi_0)A_0=
   \begin{pmatrix}
    0 & 0 \\
    0 & S_0\\
    \end{pmatrix} \quad {\rm on}\ \, \partial\Delta,
\end{equation}
where $D\subset\mathbb C^n$ is a bounded strongly pseudoconvex domain  with $C^2$-smooth boundary such that
   $$
   \Delta\times\{0\}\subset D\subset\Delta\times\mathbb C^{n-1},
   $$
$A_0:=F'(\,\cdot\,,0)$ on $\overline{\Delta}$ and  $S_0\!:\partial\Delta\to {\rm Sym}(n-1,\,\mathbb C)
\footnote{${\rm Sym}(n, \mathbb C)$ denotes the set of all complex symmetric  $n\times n$ matrices.}$
is $C^2$-smooth with operator norm less than one. Then we have
\begin{equation}\label{eq:geodesic-family}
\left\{
  \begin{aligned}
    &\,\rho\circ\varphi_{\tau\hat{v}}=0,\\
    &\,\Pi\bigg(\frac{[A^t_0\,\rho_{z}\circ\varphi_{\tau\hat{v}}]}
                     {(A^t_0\,\rho_{z}\circ\varphi_{\tau\hat{v}})_1}\bigg)=0,\\
    &\,\varphi_{\tau\hat{v}}(1)=p_0,\\ &\,\varphi_{\tau\hat{v}}'(1)=
        \sqrt{1-\tau^2|\hat{v}|^2}(\sqrt{1-\tau^2|\hat{v}|^2},\,\tau \hat{v}^t)^t,\\
    &\,\left.\frac{d}{d\theta}\right|_{\theta=0}|\varphi^{\ast}_{\tau\hat{v}}(e^{i\theta})|=0
  \end{aligned}
\right.
\end{equation}
for all $|\tau|<<1$. Recall that, as we assumed before,  $\nu_{p_0}=e_1$. Combining this with \eqref{Straightening1} and \eqref{Straightening2}, we conclude that $\rho_z(p_0)=e_1$ and both the first column of $A_0(1)$ and that of its transpose $A_0^t(1)$ are equal to $e_1$. Note also that $\varphi_{\tau\hat{v}}(1)=p_0$ and
  \begin{equation*}
   |\rho_z\circ\varphi_{\tau\hat{v}}|=|\varphi^{\ast}_{\tau\hat{v}}|
   |\langle\rho_z\circ\varphi_{\tau\hat{v}},\, \overline{\varphi'_{\tau\hat{v}}}\rangle| \quad {\rm on}\ \, \partial\Delta.
  \end{equation*}
A straightforward calculation then shows that the last identity in \eqref{eq:geodesic-family} is equivalent to
\begin{equation*}
\begin{split}
\langle \varphi'_{\tau\hat{v}}(1)&,\, e_1\rangle{\rm Im}\big\langle \rho_{zz}(p_0)\varphi'_{\tau\hat{v}}(1)-\rho_{z\bar{z}}(p_0)\overline{\varphi'_{\tau\hat{v}}(1)},\, e_1\big\rangle\\
&={\rm Im}\big(\langle \varphi''_{\tau\hat{v}}(1),\, e_1\rangle+ \langle\rho_{zz}(p_0)\varphi'_{\tau\hat{v}}(1),\, \overline{\varphi'_{\tau\hat{v}}(1)}\rangle\big).
\end{split}
\end{equation*}
Differentiating this identity and those in \eqref{eq:geodesic-family} (except the last one) with respect to $\tau$  and evaluating at $\tau=0$, we obtain
\begin{equation}\label{eq:injectivity1}
   {\rm Re}\,(\psi_1/{\rm Id_{\partial\Delta}})=0,
\end{equation}
\begin{equation}\label{eq:injectivity2}
   \Pi\bigg({\rm Id_{\partial\Delta}}\big[A^t_0(\rho_{zz}\circ\varphi_0)A_0\psi
            +A^t_0(\rho_{z\bar{z}}\circ\varphi_0)\overline{A_0}\,\overline{\psi}\,\big]\bigg)=0,
\end{equation}
\begin{equation}\label{eq:injectivity3}
   \psi(1)=0,
\end{equation}
\begin{equation}\label{eq:injectivity4}
   \psi'(1)=A_0^{-1}(1)\begin{pmatrix}
    \,0\,\\
    \,\hat{v}\,
\end{pmatrix},
\end{equation}
and
\begin{equation}\label{eq:injectivity5}
{\rm Im}\Big(\langle\varphi''(1),\, e_1\rangle+\big\langle \rho_{zz}(p_0)\varphi'(1)+\rho_{z\bar{z}}(p_0)\overline{\varphi'(1)},\, e_1\big\rangle\Big)=0,
\end{equation}
where $\psi:=A^{-1}_0\varphi$.

We are now in a position to show that $\psi\equiv0$. First consider the last $(n-1)$ components $[\psi]$ of $\psi$. In view of \eqref{Straightening3}, \eqref{eq:injectivity2} means more explicitly that
   $$
   \overline{[\psi]/{\rm Id_{\partial\Delta}}}+({\rm Id^2_{\partial\Delta}}S_0)[\psi]/{\rm Id_{\partial\Delta}}\in\mathcal{O}^2(\partial\Delta,\,\mathbb C^{n-1}).
   $$
Clearly $[\psi](1)=0$. Observing also that $\varphi'_0$ is the first column of $A_0=F'(\,\cdot\,,0)$ (see \eqref{Straightening1}), we see that \eqref{equi-vanishing} is equivalent to
 \begin{equation}\label{equi-vanishing1}
   \psi(\zeta_0)=-\lambda e_1,
 \end{equation}
and hence $[\psi](\zeta_0)=0$ as well. Then the uniqueness part of Theorem \ref{thm:RH-prob2} implies $[\psi]\equiv 0$, as desired.  We next show that the same is true for the first component $\psi_1$ of $\psi$. Observe that  \eqref{eq:injectivity1}, \eqref{eq:injectivity3} and \eqref{eq:injectivity4} together imply
  \begin{equation}\label{first-comp}
     \psi_1=ic\,(1-{\rm Id}_{\overline{\Delta}})^2
  \end{equation}
with $c$ being a real constant (depending on $\hat{v}$). Indeed, \eqref{eq:injectivity1} is equivalent to
\begin{equation*}\label{eq:injectivity1'}
  {\rm Re}\big(\psi_1/(1-{\rm Id}_{\overline{\Delta}})^2\big)=0
  \quad {\rm on}\ \, \partial\Delta\!\setminus\!\{1\}.
\end{equation*}
Since also $\psi_1(1)=\psi'_1(1)=0$, the function $\psi_1/(1-{\rm Id}_{\overline{\Delta}})^2$ belongs to the Hardy space $H^1(\Delta)$ and then \eqref{first-comp} follows immediately. It remains to show that the constant in \eqref{first-comp} is actually zero.  For this we first deduce from \eqref{Straightening1} and \eqref{Straightening2} that
   \begin{equation}\label{eq:dual-formula}
     \left.\varphi^{\ast}_0\right|_{\partial\Delta}={\rm Id_{\partial\Delta}}\,\rho_{z}\circ\varphi_0,
   \end{equation}
and hence $A_0^t\varphi^{\ast}_0=e_1$ (i.e., $\varphi^{\ast}_0$ is precisely the first column of $(A^t_0)^{-1}$). Then, by differentiating the identity $\varphi=A_0\psi$ twice and evaluating at $1$ we get
  $$
  \varphi''(1)=A_0(1)\big(\psi''(1)-2(A_0^{-1})'(1)\varphi'(1)\big)
  $$
and consequently
  $$
  \varphi_1''(1)=
  \big\langle\psi''(1)-2(A_0^{-1})'(1)\varphi'(1),\, e_1\big\rangle
  =\psi''_1(1)-2\langle (\varphi^{\ast}_0)'(1),\, \overline{\varphi'(1)}\rangle.
  $$
On the other hand, by \eqref{eq:dual-formula} we have
   $$
   (\varphi^{\ast}_0)'(1)=-i\left.\frac{d}{d\theta}\right|_{\theta=0}\varphi^{\ast}_0(e^{i\theta})
   =e_1+\rho_{zz}(p_0)e_1-\rho_{z\bar{z}}(p_0)e_1.
   $$
Combining the last two equalities with \eqref{eq:injectivity5} and \eqref{first-comp}, we arrive at
\begin{equation*}\label{eq:second-der1}
  c=\frac12{\rm Im}\big\langle \rho_{zz}(p_0)e_1-\rho_{z\bar{z}}(p_0)e_1,\,
   \overline{\varphi'(1)}\big\rangle=-\frac12{\rm Im}\langle e_1, (0, \hat{v}^t)^t
   \rangle=0.
\end{equation*}
Here the second equality follows from \eqref{Straightening3}, \eqref{eq:injectivity4} and the aforementioned fact that both the first column of $A_0(1)$ and that of its transpose $A_0^t(1)$ are equal to $e_1$.

Now we have $\psi\equiv0$. Together with \eqref{eq:injectivity4} and \eqref{equi-vanishing1}, this in turn implies $\hat{v}=0$ and $\lambda=0$, thereby completing the proof.

\appendix
\section{Canonical coordinates and defining functions along a geodesic disc}\label{supplement}
\setcounter{equation}{0}
\renewcommand\theequation{A.\arabic{equation}}

In this appendix we state and prove the following result, which is essentially due to Lempert and plays a crucial role in the proof of Theorem \ref{thm:regularity-BSR}.

\begin{theorem}[cf. \cite{Lempert81, Lempert84}]\label{thmA:normalization}
Suppose $\Omega\subset\mathbb C^n,\,n>1$, is a bounded strongly linearly convex domain with $C^{m,\,\alpha}$-smooth boundary, where $m\geq 4$ and $0<\alpha<1$. Given  a complex geodesic $\varphi$ of $\Omega$,  there is a bounded strongly pseudoconvex domain $D\subset\mathbb C^n$ with $C^{m-2,\,2\alpha}$- {\rm(}resp. $C^{m-1}$-{\rm)} smooth boundary and a biholomorphism $F\!: D\to \Omega$ which extends to a $C^{m-2,\,2\alpha}$- {\rm(}resp. $C^{m-1}$-{\rm)} smooth diffeomorphism between $\overline{D}$ and $\overline{\Omega}$ when $0<\alpha\leq1 /2$ {\rm(}resp. $1/2<\alpha<1${\rm)}, such that the following hold:
\begin{enumerate}[leftmargin=2.0pc, parsep=4pt]
\item  [{\rm (i)}]
    $\Delta\times\{0\}\subset D\subset\Delta\times\mathbb C^{n-1}$;
\item  [{\rm (ii)}]
    $F(\,\cdot\,,0)=\varphi$ on $\Delta$; and
\item  [{\rm (iii)}]
    $\partial D$ admits a $C^{m-2,\,\alpha}$-defining function $\rho\!:\overline{\Delta}\times
    \mathbb C^{n-1}\to \mathbb R$ satisfying
      \begin{equation}\label{asy-expansion}
       \rho(z)=-1+|z|^2+{\rm Re}\bigg(\sum_{j,\, k=2}^n\frac{\partial^2 \rho}{\partial z_j
               \partial z_k}(z_1, 0)z_jz_k\bigg)+O\big({\rm dist}^{2+\alpha}(z, \partial\Delta
               \times\{0\})\big)
      \end{equation}
    as $\overline{D}\ni z\to\partial\Delta\times\{0\}$, where
      \begin{equation}\label{partial-regularity}
        \frac{\partial^2 \rho}{\partial z_j\partial z_k}(\,\cdot\,,0)\in C^{m-2,\,\alpha}
        (\overline{\Delta}),\quad j,\, k=2, \ldots, n
      \end{equation}
    satisfy
      \begin{equation}\label{partial-convexity}
        \bigg|\sum_{j,\, k=2}^n\frac{\partial^2 \rho}{\partial z_j\partial z_k}(z_1, 0)v_jv_k\bigg|<|v|^2
      \end{equation}
    for all $z_1\in\partial\Delta$ and $v=(v_2,\ldots,v_n)\in \mathbb C^{n-1}\!\setminus\!\{0\}$. In
    particular, $\partial D$ is strongly linearly convex near $\partial\Delta\times\{0\}$.
\end{enumerate}
\end{theorem}

\begin{proof}
Lempert stated the theorem in a slightly different way. Here we prove the theorem by adapting his arguments. For the sake of completeness, we will be very explicit in the presentation that follows.

Write $\varphi=(\varphi_1,\ldots,\varphi_n)$ and similarly for its dual mapping $\varphi^{\ast}$. As in \cite{Lempert81, Lempert84}, by making a linear change of coordinates we may assume that
the first two components $\varphi^{\ast}_1,\, \varphi^{\ast}_2$ of $\varphi^{\ast}$ do not vanish simultaneously on $\overline{\Delta}$. Then there exist two functions $\psi_1,\, \psi_2\in \mathcal{O}(\Delta)\cap C^{m-1,\,\alpha}(\overline{\Delta})$ such that
   $$
   \psi_1\varphi^{\ast}_1+\psi_2\varphi^{\ast}_2=1\quad {\rm on}\ \, \overline{\Delta}.
   $$
Next, consider the mapping $G=(G_1,\ldots, G_n)\!: \overline{\Delta}\times\mathbb C^{n-1}\to\mathbb C^n$ given by
\begin{align*}
\begin{split}
&\xi_1=G_1(w):=\varphi_1(w_1)-w_2\varphi^{\ast}_2(w_1)-\psi_1(w_1)\sum_{k=3}^n w_k\varphi^{\ast}_k(w_1),\\
&\xi_2=G_2(w):=\varphi_2(w_1)+w_2\varphi^{\ast}_1(w_1)-\psi_2(w_1)\sum_{k=3}^n w_k\varphi^{\ast}_k(w_1),\\
&\xi_j=G_j(w):=\varphi_j(w_1)+w_j,\quad j=3,\ldots,n.
\end{split}
\end{align*}
Obviously $G$ is holomorphic on $\Delta\times\mathbb C^{n-1}$.  Also using \cite[Remark 2.3]{Huang-Wang}  one can show that $\overline{\Omega}\subset G(\overline{\Delta}\times\mathbb C^{n-1})$, and $G$ is injective on $G^{-1}(\overline{\Omega})$ with Jacobian invertible on $\overline{\Delta}\times\mathbb C^{n-1}$ (cf. \cite[pp. 190--191]{KW13}). In other words, $\left. G\right|_{G^{-1}(\Omega)}\!:G^{-1}(\Omega)\to\Omega$ is a biholomorphism and $\left. G\right|_{G^{-1}(\overline{\Omega})}\!:G^{-1}(\overline{\Omega})\to\overline{\Omega}$ is a $C^{m-1,\,\alpha}$-smooth diffeomorphism. Moreover, $G(\,\cdot\,,0)=\varphi$ on $\overline{\Delta}$ and
  $$
  \Delta\times\{0\}\subset G^{-1}(\Omega)\subset\Delta\times\mathbb C^{n-1}.
  $$

To proceed we let $r_0\!:\mathbb C^n\to \mathbb R$ be a $C^{m,\,\alpha}$-defining function for $\Omega$ with gradient  of length two on a neighborhood of $\partial\Omega$, and set
   $$
   r(w_1, w'):=|\varphi^{\ast}(w_1)|\,r_0\circ G(w_1, w'), \quad (w_1, w')\in\overline{\Delta}\times\mathbb C^{n-1}.
   $$
In this way $r\in C^{m-1,\,\alpha}(\overline{\Delta}\times\mathbb C^{n-1})\cap C^m({\Delta}\times\mathbb C^{n-1})$, and it is a defining function for $G^{-1}(\Omega)$. Moreover, it follows from the definitions of $G$ and $\varphi^{\ast}$ that
   $$
   \frac{\partial r}{\partial w_j}(w_1, 0)=\overline{w}_1\delta_{1j} \quad {\rm on}\ \, \partial\Delta
   $$
for $j=1,\ldots,n$. This means that for every $w_1\in\partial\Delta$, $(w_1, 0)\in\mathbb C^n$ is the unit outward normal to the boundary of $G^{-1}(\Omega)$ at $(w_1, 0)\in\partial\Delta\times\{0\}$. As $G(w_1, w')$ is linear in $w'$, we also have
   $$
   \frac{\partial^2 r}{\partial w_i\partial \overline{w}_j}(w_1,w')
   =|\varphi^{\ast}(w_1)|\sum_{k,\, l=1}^n\frac{\partial^2 r_0}{\partial \xi_k\partial \overline{\xi}_l}\circ G(w_1, w')\frac{\partial G_k}{\partial w_i}(w_1, w')\overline{\frac{\partial G_l}{\partial w_j}(w_1, w')}
   $$
for  $i,\, j=2,\ldots,n$. In particular, this implies
   $$
   \frac{\partial^2 r}{\partial w_i\partial \overline{w}_j}\in C^{m-2,\,\alpha}(\overline{\Delta}\times\mathbb C^{n-1}),\quad i,\, j=2,\ldots, n.
   $$
On the other hand, since $G^{-1}(\Omega)$ is obviously strongly pseudoconvex, the matrix
   $$
   R(w_1):=\bigg(\frac{\partial^2 r}{\partial w_i\partial \overline{w}_j}(w_1,0)\bigg)_{2\leq i,\,j\leq n}
   $$
is positive definite  for each $w_1\in\partial\Delta$. Hence by \cite[Th\'{e}or\`{e}me B]{Lempert81}, there is a solution $H=(h_{ij})_{2\leq i,\,j\leq n}$ to the Riemann--Hilbert problem:
$$
\left\{
   \begin{array}{ll}
         H\in \mathcal{O}(\Delta)\cap C^{m-2,\,\alpha}(\overline{\Delta}, {\rm GL}(n-1, \mathbb C)), \\
         H^tR\overline{H}=I_{n-1}\quad \hbox{on $\partial\Delta$,}
   \end{array}
\right.
$$
where ${\rm GL}(n-1, \mathbb C)$ denotes the complex general linear group  and $I_{n-1}$ is the $(n-1)\times (n-1)$ identity matrix. With a slight abuse of notation, we denote by $H$ the bihilomorphism
   $$
   \Delta\times\mathbb C^{n-1}\ni(z_1, z')\mapsto(z_1, H(z_1)z')\in\Delta\times\mathbb C^{n-1}.
   $$
Set
   $$
   D=(G\circ H)^{-1}(\Omega)
     :=\left\{(z_1, z')\in \Delta\times\mathbb C^{n-1}\!: (z_1, H(z_1)z')\in G^{-1}(\Omega)\right\}
   $$
and define
   $$
   F(z_1, z'):=G\circ H(z_1, z')=G(z_1, H(z_1)z'),\quad (z_1, z')\in\overline{D}.
   $$
Then
  \begin{equation}\label{ineq: D-pinched}
   \Delta\times\{0\}\subset D\subset\Delta\times\mathbb C^{n-1},
  \end{equation}
and $F\!: D\to \Omega$ is a biholomorphism satisfying  $F(\,\cdot\,,0)=G(\,\cdot\,,0)=\varphi$ on $\Delta$. That is, (i) and (ii) in the theorem are satisfied.

Let us now look at the boundary regularity of $D$ and $F$. By construction, it is clear that $\partial D$ is $C^{m-2,\,\alpha}$-smooth and $F\!: \overline{D}\to \overline{\Omega}$ is a $C^{m-2,\,\alpha}$-smooth diffeomorphism. We now prove more, namely that $\partial D$ is $C^{m-2,\,2\alpha}$- {\rm(}resp. $C^{m-1}$-{\rm)} smooth and $F\!: \overline{D}\to \overline{\Omega}$ is a $C^{m-2,\,2\alpha}$- {\rm(}resp. $C^{m-1}$-{\rm)} smooth  diffeomorphism when $0<\alpha\leq1 /2$ {\rm(}resp. $1/2<\alpha<1${\rm)}. To this end, applying the classical Hardy--Littlewood theorem to $H\in \mathcal{O}(\Delta)\cap C^{m-2,\,\alpha}(\overline{\Delta})$ we first get a constant $C>0$ satisfying
\begin{equation}\label{growth-est}
    |H^{(m-1)}(z_1)|\leq C(1-|z_1|)^{\alpha-1}
\end{equation}
for all $z_1\in\Delta$. Next, we show that there exists a constant $0<\varepsilon<<1$ such that
\begin{equation}\label{shape-est}
    D\subset D_{\varepsilon}:=\left\{(z_1, z')\in \Delta\times\mathbb C^{n-1}\!: |z_1|^2+\varepsilon|z'|^2<1\right\}.
\footnote{After a preprint of this paper was made available on the arXiv, the author was kindly informed by W{\l}odzimierz Zwonek that the proof of \eqref{shape-est} presented here was essentially already included in the proof of \cite[Lemma 3.2]{KZ16}.}
\end{equation}
Indeed, since $G(w_1,w')$ and $H(z_1,z')$ are linear in $w'$ and $z'$ respectively,  it is  easy to check that $D$ is strongly linearly convex near $\partial \Delta\times\{0\}$. Therefore, for every $z_1\in\partial \Delta$ the function $z'\mapsto r(z_1, H(z_1)z')$ is strongly convex on a neighborhood of $0$ in $\mathbb C^{n-1}$ (note that $\{z_1\}\times\mathbb C^{n-1}$ is the complex hyperplane tangent to $\partial D$ at $(z_1, 0)$, as follows immediately from \eqref{ineq: D-pinched}). Choose a  constant $d>0$ so large that $D\subset \Delta\times B^{n-1}(0, d)$. Then there exists a small constant $\varepsilon_0>0$ such that
\begin{equation*}\label{ineq:r-est1}
    r(z_1, H(z_1)z')\geq \varepsilon_0 |z'|^2
\end{equation*}
for all $(z_1,z')\in\partial\Delta\times B^{n-1}(0, d)$. On the other hand, we can also find a constant $C_0>0$ such that
\begin{equation*}\label{ineq:r-est2}
    \left|r(z_1, H(z_1)z')-r(z_1/|z_1|, H(z_1/|z_1|)z')\right|\leq C_0(1-|z_1|)
\end{equation*}
for all $(z_1,z')\in D\!\setminus\!(\{0\}\times\mathbb C^{n-1})$. Putting these two inequalities together yields
   $$
   r(z_1, H(z_1)z')\geq \varepsilon_0 |z'|^2-C_0(1-|z_1|^2)
   $$
for all $(z_1,z')\in D$. This implies that \eqref{shape-est} holds provided $\varepsilon:=\varepsilon_0/C_0$. Now by \eqref{growth-est} and \eqref{shape-est} we obtain
   $$
   |H^{(m-1)}(z_1)z'|\leq \frac{2C}{\varepsilon^{1-\alpha}}|z'|^{2\alpha-1}
   $$
for all $(z_1, z')\in \overline{D}\!\setminus\!(\overline{\Delta}\times\{0\})$. When $0<\alpha<1 /2$, a standard argument shows that  $\left. H\right|_{\partial D_{\varepsilon}}$ is $C^{m-2,\,2\alpha}$-smooth; equivalently, so is $\left. H\right|_{\overline{D}_{\varepsilon}}$ (see, e.g., \cite[Theorem 2.3]{Kyt-BMintegral}). When $\alpha=1/2$, $H^{(m-1)}$ is bounded on $\overline{D}_{\varepsilon}$, that is, $\left. H\right|_{\overline{D}_{\varepsilon}}\in C^{m-2,\,1}(\overline{D}_{\varepsilon})$; when  $1 /2<\alpha<1$, $H^{(m-1)}$ is continuous on $\overline{D}_{\varepsilon}$, i.e., $H\in C^{m-1}(\overline{D}_{\varepsilon})$. Since the Jacobian of $H$ is invertible on $\overline{\Delta}\times\mathbb C^{n-1}$, it then follows that $\left. H\right|_{\overline{D}_{\varepsilon}}\!: \overline{D}_{\varepsilon}\to H(\overline{D}_{\varepsilon})$ is necessarily a  $C^{m-2,\,2\alpha}$- {\rm(}resp. $C^{m-1}$-{\rm)} smooth diffeomorphism when $0<\alpha\leq1 /2$ {\rm(}resp. $1/2<\alpha<1${\rm)}. In particular, this implies that
   $$
   \partial D=(\left. H\right|_{\overline{D}_{\varepsilon}})^{-1}\big(\partial (G^{-1}(\Omega))\big)
   $$
is $C^{m-2,\,2\alpha}$- {\rm(}resp. $C^{m-1}$-{\rm)} smooth and $\left. H\right|_{\overline{D}}\!: \overline{D}\to G^{-1}(\overline{\Omega})$ is a  $C^{m-2,\,2\alpha}$- {\rm(}resp. $C^{m-1}$-{\rm)}  smooth diffeomorphism as well, when $0<\alpha\leq1 /2$ {\rm(}resp. $1/2<\alpha<1${\rm)}.

Now to complete the proof, it remains to seek a defining function $\rho$ for $D$ as described in ${\rm (iii)}$. Define
  $$
  \rho(z_1,z'):=\lambda(z_1,z')r(z_1, H(z_1)z'),\quad  (z_1, z')\in \overline{\Delta}\times\mathbb C^{n-1},
  $$
where
   $$
   \lambda(z_1,z'):=1-\frac12(1-|z_1|^2)\bigg(1-\frac{\partial^2 r}{\partial w_1\partial \overline{w}_1}(z_1,0)\bigg)-2{\rm Re}\bigg(\overline{z}_1\sum_{k,\,l=2}^nz_kh_{lk}(z_1)\frac{\partial^2 r}{\partial \overline{w}_1\partial w_l}(z_1,0)\bigg).
   $$
In this way, $\rho$ is obviously a defining function for $D$ near $\partial\Delta\times\{0\}$. To  show that $\lambda,\, \rho\in C^{m-2,\,\alpha}(\overline{\Delta}\times\mathbb C^{n-1})$, it suffices to verify that
   $$
   \frac{\partial^2 r}{\partial w_1\partial \overline{w}_j}(\,\cdot\,,0)\in C^{m-2,\,\alpha}(\overline{\Delta}),\quad   j=1, \ldots, n.
   $$
For this, differentiating the identity $r(\,\cdot\,,0)=|\varphi^{\ast}|\,r_0\circ\varphi$ directly yields
\begin{align*}
\begin{split}
\frac{\partial^2 r}{\partial w_1\partial \overline{w}_1}(w_1,0)
=\,&\frac12\bigg(|(\varphi^{\ast})'(w_1)|^2-\frac{|\langle(\varphi^{\ast})'(w_1), \varphi^{\ast}(w_1)\rangle|^2}{2|\varphi^{\ast}(w_1)|^2}\bigg)\frac{r_0\circ\varphi(w_1)}{|\varphi^{\ast}(w_1)|}\\
&+{\rm Re}\bigg(\frac{\langle(\varphi^{\ast})'(w_1), \varphi^{\ast}(w_1)\rangle}{|\varphi^{\ast}(w_1)|}\sum_{k=1}^{n}\frac{\partial r_0}{\partial \overline{\xi}_k}\circ\varphi(w_1)\overline{\varphi'_k(w_1)}\bigg)\\
&+|\varphi^{\ast}(w_1)|\sum_{k,\, l=1}^n\frac{\partial^2 r_0}{\partial \xi_k\partial \overline{\xi}_l}\circ\varphi(w_1)\varphi'_k(w_1)\overline{\varphi'_l(w_1)}\\
\in\,&C^{m-2,\,\alpha}(\overline{\Delta}).
\end{split}
\end{align*}
Similarly, since  $r(w_1, w')=|\varphi^{\ast}(w_1)|\,r_0\circ G(w_1, w')$, we have
\begin{align*}
\begin{split}
\frac{\partial^2 r}{\partial w_1\partial \overline{w}_j}(w_1,w')
=\,&\frac{\langle(\varphi^{\ast})'(w_1), \varphi^{\ast}(w_1)\rangle}{2|\varphi^{\ast}(w_1)|}
\sum_{k=1}^n\frac{\partial r_0}{\partial \overline{\xi}_k}\circ G(w_1, w')\overline{\frac{\partial G_k}{\partial w_j}(w_1, w')}\\
&+|\varphi^{\ast}(w_1)|\sum_{k,\, l=1}^n\frac{\partial^2 r_0}{\partial \xi_k\partial \overline{\xi}_l}\circ G(w_1, w')\frac{\partial G_k}{\partial w_1}(w_1, w')\overline{\frac{\partial G_l}{\partial w_j}(w_1, w')}\\
\in\, & C^{m-2,\,\alpha}(\overline{\Delta}\times\mathbb C^{n-1}),\quad j=2,\ldots,n.
\end{split}
\end{align*}

We are left to verify that $\rho$ satisfies the desired property. First of all, we have by construction the following transparent equalities when $z_1\in\partial\Delta$:
\begin{itemize}[leftmargin=2.0pc, parsep=4pt]
\item  $\lambda(z_1, 0)=1$;

\item $\displaystyle{\frac{\partial \lambda}{\partial \overline{z}_1}(z_1,0)=\frac12z_1\bigg(1-\frac{\partial^2 r}{\partial w_1\partial \overline{w}_1}(z_1,0)\bigg)}$;

\item $\displaystyle{\frac{\partial \lambda}{\partial \overline{z}_j}(z_1,0)=-z_1\sum_{l=2}^n \overline{h_{lj}(z_1)}\frac{\partial^2 r}{\partial w_1\partial \overline{w}_l}(z_1,0),\quad j=2,\ldots,n.}$
\end{itemize}
Then a straightforward calculation shows that for every $z_1\in\partial\Delta$,
\begin{align*}
   &\frac{\partial \rho}{\partial z_1}(z_1, 0)=\lambda(z_1, 0)\frac{\partial r}{\partial w_1}
       (z_1, 0)=\overline{z}_1;\\
   &\frac{\partial \rho}{\partial z_j}(z_1, 0)=\lambda(z_1, 0)\sum_{k=2}^n\frac{\partial r}{\partial
       w_k}(z_1, 0)h_{kj}(z_1)=0, \quad j=2,\ldots,n;\\
   &\frac{\partial^2 \rho}{\partial z_1\partial \overline{z}_1}(z_1,0)
       =2{\rm Re}\bigg(\frac{\partial \lambda}{\partial z_1}(z_1, 0)\frac{\partial r}{\partial
         \overline{w}_1}(z_1, 0)\bigg)+\lambda(z_1, 0)\frac{\partial^2 r}{\partial w_1\partial
         \overline{w}_1}(z_1,0)\\
   &\qquad\qquad\quad\,\,\:\:=2{\rm Re}\bigg(\frac{\partial \lambda}{\partial z_1}(z_1, 0)z_1\bigg)
    +\frac{\partial^2 r}{\partial w_1\partial \overline{w}_1}(z_1,0)=1;\\
   &\frac{\partial^2 \rho}{\partial z_1\partial \overline{z}_j}(z_1,0)
      =\frac{\partial \lambda}{\partial z_1}(z_1, 0)\sum_{k=2}^n\frac{\partial r}{\partial
       \overline{w}_k}(z_1, 0)\overline{h_{kj}(z_1)}+\frac{\partial \lambda}{\partial \overline{z}_j}(z_1,
       0)\frac{\partial r}{\partial w_1}(z_1,0)\\
   &\qquad\qquad\qquad\;\;\:\:+\lambda(z_1, 0)\sum_{k=2}^n\frac{\partial^2 r}{\partial w_1\partial
        \overline{w}_k}(z_1,0)\overline{h_{kj}(z_1)}\\
   &\qquad\qquad\quad\,\,\:\:=\frac{\partial \lambda}{\partial \overline{z}_j}(z_1, 0)\overline{z}_1+\sum_{k=2}^n\frac{\partial^2 r}{\partial w_1\partial \overline{w}_k}(z_1,0)\overline{h_{kj}(z_1)}=0, \quad j=2,\ldots,n;
\end{align*}
and
\begin{align*}
   &\frac{\partial^2 \rho}{\partial z_i\partial \overline{z}_j}(z_1,0)
      =\frac{\partial \lambda}{\partial z_i}(z_1, 0)\sum_{k=2}^n\frac{\partial r}{\partial
       \overline{w}_k}(z_1, 0)\overline{h_{kj}(z_1)}+\frac{\partial \lambda}{\partial \overline{z}_j}(z_1,
       0)\sum_{k=2}^n\frac{\partial r}{\partial w_k}(z_1,0)h_{ki}(z_1)\\
   &\qquad\qquad\qquad\;\;\:\:+\lambda(z_1, 0)\sum_{k,\, l=2}^n\frac{\partial^2 r}{\partial w_k\partial
       \overline{w}_l}(z_1,0)h_{ki}(z_1)\overline{h_{lj}(z_1)}\\
   &\qquad\qquad\quad\,\,\:\:=(H^tR\overline{H})_{ij}(z_1)=\delta_{ij},\quad i,\, j=2,\ldots,n.
\end{align*}
Furthermore, we have
\begin{align*}
\begin{split}
\frac{\partial^2 \rho}{\partial z_i\partial z_j}(z_1,0)
=\,&\frac{\partial \lambda}{\partial z_i}(z_1, 0)\sum_{k=2}^n\frac{\partial r}{\partial w_k}(z_1,
0) h_{kj}(z_1)+\frac{\partial \lambda}{\partial z_j}(z_1, 0)\sum_{k=2}^n\frac{\partial r}{\partial w_k}(z_1,0)h_{ki}(z_1)\\
&+\lambda(z_1, 0)\sum_{k,\, l=2}^n\frac{\partial^2 r}{\partial w_k\partial w_l}(z_1,0)h_{ki}(z_1)h_{lj}(z_1)\\
\in\,&C^{m-2,\,\alpha}(\overline{\Delta}),\quad i,\, j=2,\ldots,n.
\end{split}
\end{align*}
This gives \eqref{partial-regularity}. On the other hand, since $\frac{\partial r}{\partial w_1}(w_1,0)=\overline{w}_1$ for all $w_1\in\partial\Delta$, taking derivative along the tangential direction gives
   $$
   w_1^2\frac{\partial^2 r}{\partial w_1^2}(w_1,0)+\bigg(1-\frac{\partial^2 r}{\partial w_1\partial \overline{w}_1}(w_1,0)\bigg)=0
   $$
on $\partial\Delta$. The same manipulation applied to $\frac{\partial r}{\partial w_k}(\,\cdot\,,0)=0$ on $\partial\Delta$ yields
   $$
   w_1\frac{\partial^2 r}{\partial w_1\partial w_k}(w_1,0)=\overline{w}_1\frac{\partial^2 r}{\partial \overline{w}_1\partial w_k}(w_1,0)
   $$
on $\partial\Delta$ for all $2\leq k\leq n$. These identities in turn imply
\begin{align*}
\begin{split}
\frac{\partial^2 \rho}{\partial z^2_1}(z_1,0)
&=2\frac{\partial \lambda}{\partial z_1}(z_1, 0)\frac{\partial r}{\partial w_1}(z_1,
0)+\lambda(z_1, 0)\frac{\partial^2 r}{\partial w_1^2}(z_1,0)\\
&=\overline{z}^2_1\bigg(1-\frac{\partial^2 r}{\partial w_1\partial \overline{w}_1}(z_1,0)\bigg)+\frac{\partial^2 r}{\partial w_1^2}(z_1,0)=0
\end{split}
\end{align*}
and
\begin{align*}
\begin{split}
\frac{\partial^2 \rho}{\partial z_1\partial z_j}(z_1,0)
&=\frac{\partial \lambda}{\partial z_j}(z_1, 0)\frac{\partial r}{\partial w_1}(z_1,
0)+\lambda(z_1, 0)\sum_{k=2}^n\frac{\partial^2 r}{\partial w_1\partial w_k}(z_1,0)h_{kj}(z_1)\\
&=\sum_{k=2}^n\bigg(\frac{\partial^2 r}{\partial w_1\partial w_k}(z_1,0)-\overline{z}^2_1\frac{\partial^2 r}{\partial \overline{w}_1\partial w_k}(z_1,0)\bigg)h_{kj}(z_1)\\
&=0,\quad j=2,\ldots,n.
\end{split}
\end{align*}
Now by considering the Taylor expansion of $\rho$, one easily sees that $\rho$ does enjoy the desired property as in \eqref{asy-expansion}. Finally, since $\partial D$ is strongly linearly convex near $\partial\Delta\times\{0\}$, inequality \eqref{partial-convexity} follows immediately. The proof is complete.
\end{proof}

\end{document}